\documentclass[11pt,oneside]{amsart}
\usepackage{amsmath}
\usepackage{amsfonts}
\usepackage{amsthm}
\usepackage{amssymb}
\usepackage[all]{xy}
 \SelectTips{cm}{}

\title{Rigidity and exotic models for $v_1$-local $G$-equivariant stable homotopy theory}

\author{Irakli Patchkoria}
\email{irpatchk@math.uni-bonn.de}

\address{I. Patchkoria \\ Mathematisches Institut der Universit{\"a}t Bonn\\ Endenicher Allee 60\\ 53115 Bonn \\ Germany}

\author{Constanze Roitzheim}
\email{C.Roitzheim@kent.ac.uk}

\address{C. Roitzheim \\ University of Kent \\ School of Mathematics, Statistics and Actuarial Science \\ Cornwallis \\ Canterbury, CT1 7NF, UK}

\subjclass[2010]{}

\DeclareMathOperator{\Ho}{Ho}

\DeclareMathOperator{\Hom}{Hom}

\DeclareMathOperator{\Sp}{Sp}

\DeclareMathOperator{\incl}{incl}
\DeclareMathOperator{\pinch}{pinch}
\DeclareMathOperator{\Res}{Res}
\DeclareMathOperator{\Map}{Map}
\DeclareMathOperator{\proj}{proj}
\DeclareMathOperator{\Cone}{Cone}
\DeclareMathOperator{\res}{res}
\DeclareMathOperator{\tr}{tr}
\DeclareMathOperator{\Comod}{Comod}
\DeclareMathOperator{\Fun}{Fun}
\DeclareMathOperator{\Mod}{Mod}
\DeclareMathOperator{\Top}{Top}

\newcommand{\lra}{\longrightarrow}
\newcommand{\Z}{\mathbb{Z}}

\newcommand{\C}{\mathcal{C}}
\newcommand{\PP}{\mathcal{P}}
\newcommand{\FF}{\mathcal{F}}
\newcommand{\FFF}{\mathbb{F}}
\newcommand{\SpG}{\Sp_G}
\newcommand{\SpH}{\Sp_H}
\newcommand{\SpJ}{\Sp_J}
\newcommand{\LSpG}{L_{v_1}\Sp_G}
\newcommand{\LSpH}{L_{v_1}\Sp_H}
\newcommand{\LSpJ}{L_{v_1}\Sp_J}
\newcommand{\SigmaGH}{\Sigma^\infty G/H_+}
\newcommand{\SigmaGK}{\Sigma^\infty G/K_+}
\newcommand{\SigmaGL}{\Sigma^\infty G/L_+}
\newcommand{\SigmaG}{\Sigma^\infty G_+}
\newcommand{\vv}{v_1^4}
\newcommand{\GEPH}{G \ltimes_H \Sigma^{\infty}{E\PP(H)}_{+}}

\newtheorem{theorem}{Theorem}[section]
\newtheorem*{theorem*}{Theorem}
\newtheorem{proposition}[theorem]{Proposition}
\newtheorem{corollary}[theorem]{Corollary}
\newtheorem{lemma}[theorem]{Lemma}

\newtheorem{definition}[theorem]{Definition}

\newtheorem{rmk}[theorem]{Remark}

\newcommand*\cocolon{%
        \nobreak
        \mskip6mu plus1mu
        \mathpunct{}%
        \nonscript
        \mkern-\thinmuskip
        {:}%
        \mskip2mu
        \relax
}

\newcommand{\lradjunction}{\,\,\raisebox{-0.1\height}{$\overrightarrow{\longleftarrow}$}\,\,}

\begin{document}
\begin{abstract}
We prove that the $v_1$-local $G$-equivariant stable homotopy category for $G$ a finite group has a unique $G$-equivariant model at $p=2$. This means that at the prime $2$ the homotopy theory of $G$-spectra up to fixed point equivalences on $K$-theory is uniquely determined by its triangulated homotopy category and basic Mackey structure. The result combines the rigidity result for $K$-local spectra of the second author with the equivariant rigidity result for $G$-spectra of the first author. Further, when the prime $p$ is at least $5$ and does not divide the order of $G$, we provide an algebraic exotic model as well as a $G$-equivariant exotic model for the $v_1$-local $G$-equivariant stable homotopy category, showing that for primes $p \ge 5$ equivariant rigidity fails in general. 
\end{abstract}
\maketitle

\section*{Introduction}

\setcounter{section}{1}

A basic idea in topology is looking at geometric objects up to homotopy and then ask the question to what extent the original objects can be rediscovered from this homotopy information. The language of \emph{model categories} places this in a clean, formal environment: a model category is a category $\C$ with a notion of homotopy between morphisms. Thus, one can consider its \emph{homotopy category} $\Ho(\C)$, which, very roughly speaking, has the same objects as $\C$ but the morphisms are homotopy classes of the morphisms of $\C$. So, just knowing $\Ho(\C)$ as a category, to what extent can the model category $\C$ be recovered from this information?  

To make this question more workable, one considers \emph{stable} model categories. The homotopy category of a stable model category is \emph{triangulated}. The triangulated structure comes from axiomatizing cofiber sequences and encodes secondary compositions such as \emph{Toda brackets}. Schwede showed in \cite{Sch07} that all of the higher homotopy information of the stable model category of spectra $\Sp$ is contained in the triangulated structure of the stable homotopy category $\Ho(\Sp)$ alone. More precisely, he showed that when given a stable model category $\C$ such that $\Ho(\C)$ is triangulated equivalent to $\Ho(\Sp)$, then $\C$ is Quillen equivalent to $\Sp$. This phenomenon is referred to as \emph{rigidity}. In other words the stable homotopy category $\Ho(\Sp)$ is \emph{rigid}. 

The second author considered a $K$-local version of rigidity at the prime 2, showing that the higher homotopy information of $K$-local spectra $L_1\Sp$ is entirely contained in the triangulated structure of their homotopy category $\Ho(L_1\Sp)$ \cite{Roi07}. The cases $\Ho(\Sp)$ and $\Ho(L_1\Sp)$ both reduce to a core calculation involving the endomorphisms of a compact generator (in Schwede's case the stable homotopy groups of spheres $\pi_*(S^0)$, in the $K$-local case the stable homotopy groups of the $K$-local sphere $\pi_*(L_1S^0)$), although the reduction techniques are rather different- Schwede uses the Adams filtration while the $K$-local proof exploits $v_1$-periodicity. For further examples of rigidity see \cite{BarRoi14b, Hut12}. 

The $K$-local stable homotopy theory at odd primes exhibits a completely different feature. By results of Bousfield \cite{Bou85}, Franke \cite{Fra96} and \cite{Pat16}, it follows that for $p \geq 5$ the $K_{(p)}$-local stable homotopy category is \emph{not} rigid. More precisely, there is an algebraic model category $\C$ whose homotopy category is triangulated equivalent to $\Ho(L_1\Sp)$ but $\C$ is not Quillen equivalent to $L_1\Sp$. In other words $\C$ is an \emph{exotic} model for the $K_{(p)}$-local stable homotopy category. A weaker result holds for $p=3$. In this case the same happens except we do not know whether the equivalence $\Ho(\C) \sim \Ho(L_1\Sp)$ is triangulated. For further examples of exotic models see \cite{Schl, DugShi09, Pat16, P17}. 

Another instance of rigidity shows up in equivariant stable homotopy theory. Let $G$ be a finite group and $\SpG$ denote the stable model category of orthogonal $G$-spectra in the sense of \cite{ManMay}. The first author obtained in \cite{P16} the following uniqueness result at the prime 2 which is referred to as \emph{$G$-equivariant rigidity}: given a cofibrantly generated, proper, $G$-equivariant stable model category $\C$ (see Section \ref{prelim}) with a triangulated equivalence $$\Psi: \Ho(\SpG) \longrightarrow \Ho(\C)$$ and isomorphisms
$$ \Psi(\SigmaGH) \cong G/H_+ \wedge^\mathbf{L} \Psi(S^0_G) \,\,\,\mbox{for all subgroups} \,\, H \leq G ,$$
that are compatible with the restrictions, conjugations and transfers, then $\C$ is $G$-$\Top_*$-Quillen equivalent to $\SpG$. Here, the extra assumptions on compatibility with the $G$-Mackey structure are entirely reasonable as the main motivation is to compare different model categories of structured $G$-spectra. (Please note that we omit the prime $2$ from notation.)

So with results on $\Ho(L_1\Sp)$ and $\Ho(\SpG)$ it is only natural to combine those methods and ask about a ``chromatic level 1'' version of $G$-equivariant rigidity. For this, we consider  $v_1$-localisation $\LSpG$ of $\SpG$ which turns out to satisfy a range of desirable properties. This localisation (described in Section \ref{sec:localcat}) is obtained by $K$-localising (or $v_1$-localising) fixed points and it is very different from localising at equivariant $K$-theory $K_G$. Studying rigidity and exotic models for $\LSpG$ can be also viewed as a logical continuation of the work that has been done in ``chromatic level 0'', i.e. in the rational equivariant stable homotopy theory, see e.g. \cite{Gre99, GS07, Bar09, BarRoi14b, BarGreKedShi, Ked16a, Ked16}. 

The following is one of the main results of this paper:
\begin{theorem} \label{maintheo}
Let  $G$ be a finite group and $\C$ a cofibrantly generated, proper, $G$-equivariant stable model category. Further let $\LSpG$ denote the category of orthogonal $G$-spectra with the $v_1$-local model structure at the prime $2$. Suppose we are given an equivalence of triangulated categories 
\[
\Psi: \Ho(\LSpG) \longrightarrow \Ho(\C)
\]
together with isomorphisms
\[
\Psi(L_{v_1}\SigmaGH) \cong G/H_+ \wedge^\mathbf{L} \Psi(L_{v_1}S^0_G) \,\,\,\mbox{for all subgroups} \,\, H \leq G
\]
which are natural with respect to the restrictions, conjugations and transfers. Then there is a zig-zag of $G$-$\Top_*$-Quillen equivalences between $\C$ and $\LSpG$.
\end{theorem}

The proof of this theorem combines methods of \cite{Roi07} and \cite{P16} as well as requires some new mathematical input. It starts by constructing a Quillen adjunction
\[
- \wedge X: \LSpG \lradjunction \C: \Hom(X,-).
\]

Next, we  compose the left derived functor ${-\wedge^\mathbf{L} X}$ with the equivalence $\Psi^{-1}$ obtaining an exact endofunctor
\[
\mathbb{F}: \Ho(\LSpG) \xrightarrow{-\wedge^\mathbf{L} X} \Ho(\C) \xrightarrow{\Psi^{-1}} \Ho(\LSpG).
\] 
We can then reduce the question to studying $\mathbb{F}$ on the compact generators $L_{v_1}\SigmaGH$ ($H \leq G$) of $\Ho(\LSpG)$, or more precisely, to showing that the map
\[
\mathbb{F}: [L_{v_1}\SigmaGH, L_{v_1}\SigmaGK]^{G,v_1}_* \longrightarrow [\mathbb{F}(L_{v_1}\SigmaGH), \mathbb{F}(L_{v_1}\SigmaGK)]^{G,v_1}_*
\]
is an isomorphism for any subgroups $H, K \leq G$. (Here, $[-,-]^{G,v_1}_*$ denotes morphisms in the $v_1$-local $G$-equivariant stable homotopy category.)

The proof combines a variety of interesting methods- classical calculations involving $v_1$-periodicity and homotopy groups of the $K$-local sphere, mod-$2$ $K$-theory of $BG$, the model category techniques and tools from equivariant homotopy theory such as the geometric fixed points, tom Dieck splitting and double coset formula. The result provides new insight to understanding of the structure of $\Ho(\SpG)$. Furthermore, it ventures towards studying $\Ho(\SpG)$ in the context of chromatic homotopy theory.  

In the last part of the paper we deal with the case of odd primes and show that rigidity can fail in general. More precisely, following Franke, we consider the category of twisted chain complexes of $E(1)_*E(1)$-comodules. Our exotic algebraic model is the model category of twisted $([1],1)$-chain complexes of Mackey functor objects in $E(1)_*E(1)$-comodules, denoted by $$\mathcal{C}^{([1],1)}(E(1)_*E(1)[\mathcal{M}(G)]\mbox{-}\Comod).$$ The main result of the cocluding part of this paper is that if $p$ is an odd prime and $p$ does not divide the order of the finite group $G$, then $\mathcal{C}^{([1],1)}(E(1)_*E(1)[\mathcal{M}(G)]\mbox{-}\Comod)$ is Quillen equivalent to a cofibrantly generated, proper $G$-equivariant stable model category, thus showing that the analogous result to Theorem \ref{maintheo} does not hold if $p \geq 5$. The case when $p$ is an odd prime and the order of the finite group $G$ is divisible by $p$ remains open.


\bigskip
This paper is organised as follows.

We start by quickly recalling some notions and discuss notation. Then, in Section \ref{sec:localcat} we consider the $v_1$-local model structure on $G$-spectra and discuss its properties. In particular, we indicate why some of the classical results in equivariant stable homotopy theory also hold in $\Ho(\LSpG)$.

\medskip
In Section \ref{sec:quillenfunctor} we set up the Quillen functor that is going to be part of the main theorem. In \cite{P16} it was shown that any cofibrantly generated, proper, $G$-equivariant stable model category can be replaced by a $G$-spectral one. For such a $\C$ and a cofibrant $X \in \C$, there is a Quillen adjunction
\[
- \wedge X: \Sp_G: \lradjunction \C: \Hom(X,-).
\]
We discuss when this set-up factors over the localisation $L_{v_1}\Sp_G$.

\medskip
Section \ref{sec:calculation} provides the main calculation involving generators and relations of $\pi_*(L_1S^0)$. We show that an exact functor $$F : \Ho(\Sp_G) \longrightarrow \Ho(L_{v_1}\Sp_G)$$ compatible with the necessary $G$-equivariant structure sends $\Sigma^\infty G_+ \wedge v_1^4$ to an isomorphism. We also show that any exact functor $$\mathbb{F} : \Ho(\LSpG) \longrightarrow \Ho(L_{v_1}\Sp_G)$$ with certain equivariance preoperties induces an isomorphism 
\[\mathbb{F}: [L_{v_1}\SigmaG, L_{v_1}\SigmaG]^{G,v_1}_* \xrightarrow{\cong} [\mathbb{F}(L_{v_1}\SigmaG), \mathbb{F}(L_{v_1}\SigmaG)]^{G,v_1}_* \]
when $\ast=0, \cdots, 9$. This uses relations and Toda brackets in $\pi_*(L_1S^0)$ as in \cite{Roi07}.

\medskip
In Section \ref{sec:revisited} we finish the proof that
\[
- \wedge X: L_{v_1}\Sp_G \lradjunction \C: \Hom(X,-)
\]
is a Quillen functor using cellular chains, Bredon homology and the equivariant Atiyah-Hirzebruch spectral sequence. This input is also new to the rigidity context and potentially has applications also in other rigidity problems.

\medskip
In Section \ref{sec:quillenequivalence} we show that the previously constructed Quillen functor is actually a Quillen equivalence. This uses our calculation in Section \ref{sec:calculation} and some classical results in equivariant stable homotopy theory. 

\medskip
To put the rigidity result in context, in Section \ref{sec:exotic} we show that $\Ho(\LSpG)$ has an algebraic model as well as an exotic $G$-equivariant stable model at primes $p \ge 5$ with $p \nmid |G|$. 

\section*{Acknowledgments} The first author acknowledges the support of the Danish National Research Foundation through the Centre for Symmetry and Deformation (DNRF92) and the German Research Foundation Schwerpunktprogramm 1786. He also thanks the Hausdorff Research Institute for Mathematics in Bonn for their hospitality. Finally the first author would like to thank Gijs Heuts, Akhil Mathew, Stefan Schwede and Christian Wimmer for helpful conversations.

The second author thanks the Department of Mathematical Sciences in Copenhagen and the Hausdorff Research Institute for Mathematics in Bonn for their hospitality. Furthermore, the second author would like to thank David Barnes, Anna Marie Bohmann, John Greenlees and Mike Hill for useful discussions.

\section{Background and conventions} \label{prelim}

We assume that the reader is familiar with the basic background related to stable model categories and (left) Bousfield localisation, for example see \cite{Hir03, Hov99, Bou79}. 

By $L_1=L_{K_{(p)}}$ we denote Bousfield localisation with respect to $p$-local topological $K$-theory. The $1$ in the notation refers to the fact that $L_{K_{(p)}}=L_{E(1)}$ where $E(1)$ is the Johnson-Wilson spectrum of chromatic level 1, i.e. $$K_{(p)} \simeq \bigvee\limits_{i=0}^{p-2}\Sigma^{2i} E(1) \,\,\,\,\mbox{and}\,\,\,\, E(1)_*=\mathbb{Z}_{(p)}[v_1, v_1^{-1}], \,\,\,|v_1|=2(p-1).$$

Throughout the paper, $G$ is a finite group. We will use the stable model category $\SpG$ of \emph{equivariant orthogonal $G$-spectra} based on a complete universe constructed in \cite{ManMay}. This model category is enriched, tensored and cotensored over the model category of pointed $G$-spaces $G$-$\Top_*$. Moreover, for any finite dimensional orthogonal $G$-representation $V$, the functor $$S^V \wedge - : \Ho(\SpG) \xrightarrow{\sim} \Ho(\SpG)$$ is an equivalence of categories. Here $S^V$ is the representation sphere of $V$, i.e. the one-point compactification of $V$. In fact $\SpG$ is an example of a \emph{$G$-equivariant stable model category}, see e.g. \cite[Definition 3.1.1]{P16}: recall that a model category $\C$ is a $G$-equivariant stable model category if it is enriched, tensored and cotensored over the category of pointed $G$-spaces $G$-$\Top_*$ in a compatible way (i.e. the pushout-product axiom holds) and if the adjunction
$$\xymatrix{ S^V \wedge -  \colon \C \ar@<0.5ex>[r] & \C \cocolon \Omega^V(-) \ar@<0.5ex>[l]}$$
is a Quillen equivalence for any finite dimensional orthogonal $G$-representation $V$. (When we write adjunctions, the left adjoint is always the top arrow.) 

Next we also note that the model category $\C$ in Theorem \ref{maintheo} is assumed to be a cofibrantly generated $G$-$\Top_*$-model category in the sense of Definition 3.3.1 of \cite{P16}. 

We mention that all the known classical models for $G$-equivariant stable homtopy theory based on spaces (and not on simplicial sets) are cofibrantly generated, proper $G$-equivariant stable model categories. These include as already mentioned $\SpG$ \cite{ManMay}, the $\mathbb{S}$-model structure (flat model structure, see \cite[Theorem 2.3.27]{Sto} and also \cite{BrDuSt}), the model category of $S_G$-modules \cite[IV.2]{ManMay}, the model category of $G$-equivariant continuous functors \cite{Blumberg} and the model categories of $G$-equivariant topological symmetric spectra in the sense of \cite{Man04} and \cite{Haus}.

We will denote graded morphisms in $\Ho(\Sp)$ by $[-,-]^{\Sp}_*$ and morphisms in $\Ho(L_1\Sp)$ by $[-,-]^{L_1\Sp}_*$. For graded morphisms in the equivariant stable homtopy category $\Ho(\SpG)$ we will use $[-,-]^{G}_*$. 

A basic computation in equivariant stable homotopy theory needed in the formulation of the main theorem is that of $[\SigmaGH,\SigmaGK]^{G}_0$. There are three types of basic stable maps here. Let ${}^g H$ denote the conjugate subgroup $gHg^{-1}$. Then the map 
$$g \colon \Sigma^{\infty}G/{}^g H_{+} \lra \SigmaGH$$
in the homotopy category $\Ho(\SpG)$ given by $[x] \mapsto [xg]$ on the point-set level is called the \emph{conjugation} map associated to $g$ and $H$. Further, if $K$ is another subgroup of $G$ such that $K \leq H$, then we have the \emph{restriction} map
$$\res_{K}^H \colon \SigmaGK \lra \SigmaGH$$
which is just the obvious projection on the point-set level. Moreover, there is also a map backwards, called the \emph{transfer} map
$$\tr_{K}^H \colon \SigmaGH \lra \SigmaGK,$$
given by the Pontryagin-Thom construction (see e.g. \cite[IV.3]{LewMaySte86} or \cite[II.8]{tom}). These morphisms generate the full preadditive subcategory of $\Ho(\SpG)$ with objects the stable orbits $\SigmaGH$, $H \leq G$ \cite[V.9]{LewMaySte86}. In fact this category is equivalent to the \emph{Burnside category} $\mathcal{B}(G)$ of $G$ used in Section \ref{sec:exotic}.

\section{$v_1$-local $G$-spectra}\label{sec:localcat}

In this section we are going to construct the $v_1$-localisation of $G$-spectra and describe its properties. We will work at the prime $2$ (omitted from notation), though most of the constructions here work at any prime (see Remark \ref{otherprimes}).

Throughout the paper let $M$ denote the mod-$2$ Moore spectrum, which fits into the exact triangle in $\Ho(\Sp)$
\[
S^0 \stackrel{2}{\longrightarrow} S^0 \xrightarrow{\incl} M \xrightarrow{\pinch} S^1.
\]
Here, $\incl$ is the inclusion of the bottom cell, and $\pinch$ is the map that ``pinches'' off the bottom cell so that only the top cell is left.
As $M$ is rationally trivial it possesses a $v_1$-self map (see for example \cite{HopSm98})
\[
v_1^4: \Sigma^8 M \longrightarrow M
\]
which induces an isomorphism in $K$-theory. Note that despite the notation, $v_1^4$ is not the fourth power of some existing map $v_1$, it only alludes to the fact that the map has degree 8 and the element $v_1 \in E(1)_*$ has degree 2. By \cite[Theorem 4.8]{Bou79} a spectrum $X$ is $K$-local if and only if
\[
[M,X]^{\Sp}_n \xrightarrow{(v_1^4)^*} [M,X]^{\Sp}_{n+8}
\]
is an isomorphism for all $n$. Thus, we can think of $K$-localisation as ``$v_1$-localisation''. 

We now translate this into the equivariant setting.
\bigskip
\begin{definition}
The \emph{$v_1$-localisation} of $\SpG$ is defined as the left Bousfield localisation at the set of maps
\[
W=\{ \SigmaGH \wedge \Sigma^8 M \xrightarrow{\SigmaGH \wedge \vv} \SigmaGH \wedge M \,\,|\,\, H \leq G \}.
\]
The notation for the resulting model category is $\LSpG$, and the morphisms in the homotopy category $\Ho(\LSpG)$ are denoted by $[-,-]^{G,v_1}_*$.
\end{definition}

This localisation exists because $\SpG$ is proper and cellular, and $W$ is a set \cite{Hir03}. Using \cite{Hir03} we also see that this localisation is a cofibrantly generated, proper, $G$-equivariant stable model category (see \cite[Proposition 4.5.1 and Lemma 12.4.1]{Hir03} for smallness conditions). Similarly one can localise the other models of $G$-equivariant stable homotopy theory mentioned in Section \ref{prelim}. These model categories are again proper and cellular, and their localisations are cofibrantly generated, proper, $G$-equivariant stable model categories.

We note the following immediate consequences. The notation $(-)^H$ used below will stand for both derived and point-set level categorical fixed points. 

\begin{itemize}
\item $X \in \SpG$ is $v_1$-local if and only if  
\[
[\SigmaGH \wedge M, X]^{G}_n \xrightarrow{(\SigmaGH \wedge \vv)^*} [\SigmaGH \wedge M, X]^{G}_{n+8} 
\]
is an isomorphism for all $n$ and all $H \leq G$. By adjunction, this is equivalent to 
\[
[M, X^H]^{\Sp}_n \xrightarrow{(\vv)^*} [M, X^H]^{\Sp}_{n+8}
\]
being an isomorphism for all $n$ and all $H \leq G$. By the properties of $\vv$ mentioned earlier, this is equivalent to the derived fixed points $X^H$ being $K$-local for all $H \leq G$. 
\item The Quillen adjunction $$ \SigmaGH \wedge -: \Sp \lradjunction \SpG: (-)^H$$ is also a Quillen adjunction $$ \SigmaGH \wedge -: L_1\Sp \lradjunction \LSpG: (-)^H$$ because $(-)^H$ sends $v_1$-local $G$-spectra to $K$-local spectra. (Note that in general, a Quillen adjunction  $$F: \mathcal{C} \lradjunction \mathcal{D}: G$$  is also a Quillen adjunction $F:L_S \mathcal{C} \lradjunction \mathcal{D}: G$ if and only if $G$ sends fibrant objects in $\mathcal{D}$ to $S$-local objects in $\mathcal{C}$, see \cite{BarRoi11b}. We prove this explicitly in Lemma \ref{lem:factor}.)
\end{itemize}

\begin{rmk} When we work with model categories and say that an object $X$ is local, then we will assume that it underlying fibrant. However, when working in homotopy categories we will ignore this condition. If we need to make $X$ underlying fibrant, this can be always arranged by replacing it fibrantly in the original model category. Further, in the context of homotopy categories $X^H$ will always stand for derived fixed points, and when considering fixed points on the point-set level, then $X^H$ will denote the (non-derived) categorical fixed points.
 
\end{rmk}

The set $\{ \SigmaGH \,\,|\,\, H \leq G\}$ is a set of compact generators for $\Ho(\SpG)$. We now describe the situation for $\Ho(\LSpG)$. 

\begin{proposition} \label{generators}
The localisation $L_{v_1}$ is smashing. In particular, 
\[
\{ L_{v_1}\SigmaGH \,\,|\,\, H \leq G\}
\]
is a set of compact generators for $\Ho(\LSpG)$.
\end{proposition}

\begin{proof}

The proof is based on results of \cite{HPS}. We start by observing that an arbitrary coproduct of $v_1$-local objects is $v_1$-local. This follows from the fact that the domains and codomains of morphisms in $W$ are compact in $\Ho(\SpG)$. Reformulating, we see that $L_{v_1}$ preserves arbitrary coproducts. Hence the localisation is smashing by \cite[Definition 3.3.2]{HPS} and $L_{v_1}$ preserves compactness by \cite[Theorem 3.5.2]{HPS}.\end{proof}



We can now describe the $v_1$-local $G$-sphere.

\begin{proposition} \label{localsphere}
Let $\varepsilon^*: \Ho(\Sp) \longrightarrow \Ho(\SpG)$ be the derived inflation functor, and let $L_1S^0$ be the $K$-local sphere. Then $\varepsilon^* L_1S^0$ is the $v_1$-localisation of the $G$-equivariant sphere spectrum $S^0_G$. 
\end{proposition}

\begin{proof} We have to show the following.
\begin{enumerate}
\item $S^0_G=\varepsilon^* S^0 \longrightarrow \varepsilon^* L_1S^0$ is a $v_1$-local equivalence.
\item $\varepsilon^* L_1S^0$ is a $v_1$-local $G$-spectrum. 
\end{enumerate}
For the first point, let $Z$ be a $v_1$-local $G$-spectrum. This implies that $Z^G$ is a $K$-local spectrum. As $S^0 \longrightarrow L_1S^0$ is a $K$-equivalence, this means that the induced map
\[
[L_1S^0, Z^G]^{\Sp}_* \longrightarrow [S^0,Z^G]^{\Sp}_*
\]
is an isomorphism. By adjunction,
\[
[\varepsilon^* L_1S^0, Z]^{G}_* \longrightarrow [\varepsilon^* S^0,Z]^{G}_*
\]
is an isomophism. This implies that $S^0_G=\varepsilon^*S^0 \longrightarrow \varepsilon^*L_1S^0$ is a $v_1$-local equivalence.

The $G$-spectrum $\varepsilon^* L_1S^0$ is $v_1$-local if
\[
[\SigmaGH \wedge M, \varepsilon^* L_1S^0]^G_* \xrightarrow{(\SigmaGH \wedge \vv)^*} [\SigmaGH \wedge M, \varepsilon^* L_1S^0]^G_{*+8}
\]
is an isomorphism. By Spanier-Whitehead duality (denoted by $D$), this is equivalent to 
\[
[\SigmaGH, DM \wedge \varepsilon^* L_1S^0]^G_* \xrightarrow{{D(\vv)}_*} [\SigmaGH, DM \wedge \varepsilon^* L_1S^0]^G_{*-8}
\]
being an isomorphism. As $\varepsilon^*$ is monoidal, this is the same as asking for 
\[
[\SigmaGH, \varepsilon^*(DM \wedge L_1S^0)]^G_* \xrightarrow{{D(\vv)}_*} [\SigmaGH,  \varepsilon^*(DM \wedge L_1S^0)]^G_{*-8}
\]
to be an isomorphism. But $$D(\vv) \wedge L_1S^0: DM \wedge L_1S^0 \longrightarrow \Sigma^{-8} DM \wedge L_1S^0$$ is an isomorphism in $\Ho(\Sp)$, so we can conclude that $\varepsilon^* L_1S^0$ is a $v_1$-local $G$-spectrum as desired.
\end{proof}

\begin{corollary}
For a $G$-spectrum $X$, $L_{v_1}X \cong X \wedge L_{v_1}S^0_G \cong X \wedge \varepsilon^* L_1S^0$ in $\Ho(\SpG)$. \qed
\end{corollary}

From now on when working in $\Ho(\LSpG)$ we will sometimes identify $X$ with $L_{v_1}X$ and $X \wedge L_{v_1}S^0_G$ depending on context and convenience.

Below we list some classical facts from equivariant stable homotopy theory. The references are for example \cite{LewMaySte86, ManMay, HHR, Sch16}. Because of the good properties of $v_1$-localisation these facts just pass to the $v_1$-local $G$-equivariant stable homotopy theory.

Recall that $\Phi^G: \Ho(\SpG) \rightarrow \Ho(\Sp)$ denotes the (derived) geometric fixed points functor (see e.g. \cite[Section V.4]{ManMay}). We will use that $\Phi^G \circ \varepsilon^*$ is isomorphic to the identity functor and $\Phi^G$ is monoidal. An immediate consequence of these properties and the previous result is the following:  

\begin{corollary} \label{geofixedpt} For any subgroup $H$ of $G$, there is an isomorphism in $\Ho(\Sp)$
$$\Phi^H(L_{v_1}X)\cong L_1 \Phi^H(X).$$ \qed \end{corollary}

Finally, for genuine  fixed points we have:
\begin{corollary} \label{fixedpt} For any subgroup $H$ of $G$, there is an isomophism in $\Ho(\Sp)$
$$(L_{v_1}X)^H  \cong L_1(X^H).$$ \end{corollary}

\begin{proof} The proof uses Proposition \ref{generators} and Proposition \ref{localsphere}. It also uses the isomorphism $(Z \wedge \varepsilon^*Y)^H \cong Z^H \wedge Y$ in $\Ho(\Sp)$ for any spectrum $Y$ and an $H$-spectrum $Z$. \end{proof}

Next, we list the following facts:

\medskip
{\bf $v_1$-local tom Dieck splitting.} For a based $G$-CW complex $X$, the tom Dieck splitting says that 
\[
(\Sigma^\infty X)^G \cong \bigvee_{(H) \leq G} \Sigma^\infty(EW_G(H)_+ \wedge_{W_G(H)} X^H)
\]
in $\Ho(\Sp)$, where $W_G(H)$ is the Weyl group of $H$ in $G$. Because of the compatibility of localisation with fixed points and because $L_1$ is smashing, this gives us a chain of isomorphisms in $\Ho(L_1\Sp)$ 
\begin{eqnarray}
(L_{v_1}\Sigma^\infty X)^G \cong L_1((\Sigma^\infty X)^G) \cong L_1 (\bigvee_{(H) \leq G} \Sigma^\infty(EW_G(H)_+ \wedge_{W_G(H)} X^H)) \nonumber\\ \cong \bigvee_{(H) \leq G}L_1( \Sigma^\infty(EW_G(H)_+ \wedge_{W_G(H)} X^H)). \nonumber
\end{eqnarray}


\medskip
{\bf Restriction and external transfer.} Here we use Lemma \ref{lem:factor}. The Quillen adjunction 
\[
G \ltimes_H -: \SpH \lradjunction \SpG: \Res^G_H
\]
is also a Quillen adjunction 
\[
G \ltimes_H -: \LSpH \lradjunction \LSpG: \Res^G_H.
\]
The restriction $\Res^G_H$ sends $v_1$-local spectra to $v_1$-local spectra by definition.
Similarily, 
\[
 \Res^G_H: \LSpG \lradjunction \LSpH: \Map_H(G_+,-)
 \]
 is a Quillen adunction. For this, we have to show that $\Map_H(G_+, L_{v_1}X)$ is $v_1$-local, i.e. 
\[ 
 [\SigmaGH \wedge V(1), \Map_H(G_+, L_{v_1}X)]^G_*=0,
\] 
 where $V(1)$ is the cofibre of $\vv$. But 
\begin{eqnarray}
[\SigmaGH \wedge V(1), \Map_H(G_+, L_{v_1}X)]^G_* &\cong&[\Res^G_H(\SigmaGH \wedge V(1)), L_{v_1}X]^H_*,\nonumber\\& \nonumber
\end{eqnarray}
which is trivial by the double coset formula and because $L_{v_1}X$ is a $v_1$-local $H$-spectrum. 

\medskip
{\bf $v_1$-local Wirthm{\"u}ller isomorphism.}
The classical Wirthm{\"u}ller isomorphism asserts that the canonical map $$w_H(X): G \ltimes_H X \xrightarrow{\sim} \Map_H(G_+,X)$$ is a weak equivalence. In other words, the left and right adjoints of the restriction on the homotopy category are isomorphic. This passes through the $v_1$-local categories since any weak equivalence is in particular a $v_1$-local equivalence.

\medskip
{\bf $v_1$-local double coset formula.} The classical double coset formula is the isomophism
$$\bigvee_{[g] \in K \setminus G / H} K \ltimes_{{}^g H \cap K } \Res^{{}^g H}_{{}^g H \cap K} (c_g^*X) \xrightarrow{\cong} \Res^G_K (G \ltimes_H X),$$
which also passes to the $v_1$-local categories. Here $c_g$ is the conjugation map $g^{-1}(-)g$ for $g \in G$.

\medskip
{\bf $v_1$-local inflation.} A short exact sequence of groups 
\[
1 \longrightarrow N \xrightarrow{\iota} G \xrightarrow{\varepsilon} J \longrightarrow 1
\]
induces a Quillen adjunction 
\[
\varepsilon^*: \SpJ \lradjunction \SpG: {(-)}^N.
\]
It also induces a Quillen adjunction 
\[
\varepsilon^*: \LSpJ \lradjunction \LSpG: {(-)}^N
\]
because ${(-)}^N$ sends $v_1$-local $G$-spectra to $v_1$-local $J$-spectra (Corollary \ref{fixedpt}). We will often consider non-equivariant spectra as $G$-equivariant spectra via the inflation functor $\varepsilon^*: \Ho(\Sp) \longrightarrow \Ho(\SpG)$ and somtimes skip the symbol $\varepsilon^*$.

We also note that Corollary \ref{geofixedpt} implies that there is a $v_1$-local geometric fixed points functor $$\Phi^N \colon \Ho(\LSpG) \rightarrow \Ho(\LSpJ)$$ and a triangulated natural isomorphism $$id \cong \Phi^N\circ\varepsilon^*$$ of endofunctors on $\Ho(\LSpJ).$

\medskip
{\bf Families and $v_1$-localisation.}
Let $\mathcal{F}$ be a family of subgroups of $G$. There is a stable model category of $G$-spectra $\Sp_{G,\mathcal{F}}$ based on the family $\mathcal{F}$ \cite[Section IV.6]{ManMay} which comes with a Quillen adjunction 
\[id: \Sp_{G,\mathcal{F}} \lradjunction \SpG: id.\]
Using Corollary \ref{fixedpt} and Lemma \ref{lem:factor} we see that this adjunction passes through $v_1$-localisations.



 \begin{rmk} \label{otherprimes}
 The constructions and results of this section do not depend on the prime 2. An analogous construction can be performed at odd primes using $$v_1: \Sigma^{2p-2}M \longrightarrow M$$ instead of $v_1^4$, and the same properties hold. We are using this in Section \ref{sec:exotic}.
 \end{rmk}

\section{The Quillen functor}\label{sec:quillenfunctor}

In order to obtain the Quillen equivalence of our main theorem, we first need a Quillen functor of which we will show in Section \ref{sec:quillenequivalence} that it is in fact an equivalence. The first author showed in \cite[Proposition 3.1.2]{P16} that without loss of generality, our (proper, cofibrantly generated) $G$-equivariant stable model category $\mathcal{C}$ can be replaced by a Quillen equivalent $G$-spectral model category, i.e. by an $\SpG$-model category in the sense of \cite{Dug01}. So from now on, let us assume $\mathcal{C}$ to be a proper, cofibrantly generated $\SpG$-model category. This means that for any cofibrant $X \in \LSpG$, the $G$-spectral structure provides us with a Quillen functor
\[
- \wedge X: \SpG \lradjunction \mathcal{C}: \Hom(X,-).
\]
The model category $\LSpG$ has the same underlying category as $\SpG$ but a different model structure, so it is not clear if 
\[
- \wedge X: \LSpG \lradjunction \mathcal{C}: \Hom(X,-).
\]
is also a Quillen adjunction. For this we will need the following lemma.

\begin{lemma}\label{lem:factor}
Let $G: \mathcal{M} \longrightarrow \mathcal{N}$ be a right Quillen functor of model categories $\mathcal{M}$ and $\mathcal{N}$, and let $L_S \mathcal{N}$ be a left Bousfield localisation of $\mathcal{N}$. Then $G: \mathcal{M} \longrightarrow L_S\mathcal{N}$ is a right Quillen functor if and only if $G$ sends fibrant objects to $S$-local objects. 
\end{lemma}

\begin{proof}
It is clear that $G$ sends acyclic fibrations to acyclic fibrations as $\mathcal{N}$ and $L_S\mathcal{N}$ have the same cofibrations. 
If $X \in \mathcal{M}$ is fibrant, then $G(X)$ is fibrant. Also, $G(X)$ is $S$-local by assumption, so $G(X)$ is fibrant in $L_S\mathcal{N}$. We still have to show that $G$ sends fibrations to fibrations in $L_S\mathcal{N}$. So let $g: X \longrightarrow Y$ be a fibration in $\mathcal{M}$ between fibrant objects. This implies that $G(g): G(X) \longrightarrow G(Y)$ is a fibration in $\mathcal{N}$ between $L_S\mathcal{N}$-fibrant objects. We can factor $G(g)$ in $L_S\mathcal{N}$ as acyclic cofibration followed by fibration
\[
\xymatrix{ G(X) \ar@{>->}[r]^-{i}_-{\sim} & C \ar@{->>}[r]^-{p} & G(Y).
}
\]
As $C$ and $G(X)$ are fibrant in $L_S \mathcal{N}$ and $i$ is an $S$-equivalence, $i$ is also a weak equivalence in $\mathcal{N}$. As it is also a cofibration, it is thus an acyclic cofibration in $\mathcal{N}$. This implies that there is a lift $H$ in the following diagram
\[
\xymatrix{ G(X) \ar@{=}[r]\ar@{>->}[d]_{i} & G(X) \ar[d]^{G(g)} \\
C \ar[r]_{p}\ar@{.>}[ur]^{H} & G(Y).
}
\]
This means that there is a retract in $L_S\mathcal{N}$
\[
\xymatrix{ G(X) \ar[d]_{G(g)} \ar[r]^{i}& C\ar@{->>}[d]_{p}\ar[r]^{H} & G(X) \ar[d]^{G(g)} \\
G(Y)\ar@{=}[r] & G(Y) \ar@{=}[r]& G(Y).
}
\]
In other words, $G(g)$ is a retract of a fibration in $L_S\mathcal{N}$ and thus a fibration itself, which is what we wanted to prove. By \cite[A.2]{Dug01}, a functor that preserves trivial fibrations as well as fibrations between fibrant objects is a right Quillen functor.
\end{proof}

\begin{corollary}
The adjunction
\[
- \wedge X: \LSpG \lradjunction \mathcal{C}: \Hom(X,-).
\]
is a Quillen adjunction if and only if $\mathbf{R}\Hom(X,Y)$ is a $v_1$-local $G$-spectrum for all $Y \in \mathcal{C}$. 
\end{corollary}

Recall that we assume an equivalence of triangulated categories
\[
\Psi: \Ho(\LSpG) \longrightarrow \Ho(\mathcal{C})
\]
together with isomorphisms
\[
\Psi(L_{v_1}\SigmaGH) \cong G/H_+ \wedge^\mathbf{L} \Psi(L_{v_1}S^0_G) \,\,\,\mbox{for all subgroups} \,\, H \leq G
\]
which are natural with respect to the restrictions, conjugations and transfers. From now on, let $X$ be a fibrant and cofibrant replacement of $\Psi(L_{v_1}S^0_G)$ in $\mathcal{C}$. We want to show the following.

\begin{proposition}\label{prop:mainprop}
Let $X$ be as above. Then the $G$-spectrum $\mathbf{R}\Hom(X,Y)$ is $v_1$-local for all $Y \in \C$. 
\end{proposition}

We will perform some essential calculations in a separate section before proving this.

\section{The calculation}\label{sec:calculation}

This section deals with the computational part of our main argument. First, we need to recall the stable homotopy groups of the $2$-local sphere and of the $K$-local sphere  at the prime $2$ in low dimensions and their generators.

\begin{tabular}[h]{c | c |c}\label{table}
$k$ & $\pi_kS^0_{(2)}$ &  $\pi_k L_1 S^0$  \\ \hline
$0$ & $\mathbb{Z}_{(2)}\{\iota\}$ & $\mathbb{Z}_{(2)}\{\iota\} \oplus \mathbb{Z}/2\{y_0\}$ \\
$1$ &  $\mathbb{Z}/2\{\eta\}$ & $\mathbb{Z}/2\{\eta,y_1\}$ \\
$2$ &  $\mathbb{Z}/2\{\eta^2\}$ & $\mathbb{Z}/2\{\eta^2\}$ \\
$3$ & $\mathbb{Z}/8\{\nu\}$  & $\mathbb{Z}/8\{\nu\}$ \\
$4$ & 0  & 0 \\
$5$ & 0 & 0 \\
$6$ &  $\mathbb{Z}/2\{\nu^2\}$ & 0 \\
$7$ & $\mathbb{Z}/16\{\sigma\}$ & $\mathbb{Z}/16\{\sigma\}$\\
$8$ & $\mathbb{Z}/2\{\eta\sigma, \varepsilon\}$ & $\mathbb{Z}/2\{\eta\sigma\}$\\
$9$ & $\mathbb{Z}/2\{\eta^2\sigma, \eta\varepsilon, \mu\}$ & $\mathbb{Z}/2\{\eta^2\sigma,\mu\}$\\
\end{tabular}

We want to study the behaviour of these elements under certain exact functors
\[
F: \Ho(\SpG) \longrightarrow  \Ho(\LSpG) \,\,\,\mbox{and}\,\,\, \mathbb{F}: \Ho(\LSpG) \longrightarrow \Ho(\LSpG).
\]

\bigskip
We also make use of the relations
\begin{equation}\label{relations}
4\nu = \eta^3, \, y_1=\eta y_0, \,\eta y_1 = 0, \,y_0^2=0, \,y_1^2=0, \,\sigma y_1=0 \quad\quad\mbox{and}\quad \mu y_0 = \eta^2 \sigma
\end{equation}
(see \cite[8.15.(d)]{Rav84} and \cite[Corollary 4.5]{Bou79}) and the Toda brackets
\begin{eqnarray}
8\sigma =  \left<\nu,8,\nu\right>   & &   \nonumber\\
\mu \in  \left<2,8\sigma,\eta\right> & & \quad(\mbox{indeterminacy: } \eta^2\sigma, \eta \varepsilon). \nonumber \\
\end{eqnarray}

Note that $K$-locally the latter indeterminacy is just $\eta^2\sigma$, since $\varepsilon$ does not survive $K$-localisation.  
For further references, see
\cite{Tod62}, Lemma 5.13 and the tables in Chapter XIV. For more details on the endomorphisms of the $K$-local mod-2 Moore spectrum and their relations we also refer to \cite[Section 5]{Roi07}.

The following is well-known to experts but we are not aware of a reference in literature, so we provide a proof.

\begin{lemma} \label{v1relation} The composite
\[
\Sigma^8 M \xrightarrow{\pinch}  S^9 \xrightarrow{^\mu}  S^0 \xrightarrow{\incl}  M
\]
is equal to $2v_1^4$. 
 
\end{lemma}

\begin{proof} We know from \cite[Lemma 12.5 and proof of Theorem 12.13]{Ada66} that the composite
\[ 
S^8 \xrightarrow{\incl}  \Sigma^8 M \xrightarrow{v_1^4}  M \xrightarrow{\pinch}  S^1  
\]
is equal to $8 \sigma$. Using this and the Toda bracket 
\[\mu \in  \left<2,8\sigma,\eta\right>  \quad(\mbox{indeterminacy: } \eta^2\sigma, \eta \varepsilon)\]
we get a commutative diagram
\[\xymatrix{S^8 \ar[r]^2 & S^8 \ar[rr]^{8\sigma} \ar[dd]_{\incl} & & S^1 \ar[r]^\eta & S^0\\ & & M \ar[ur]^{\pinch} & &  \\ & \Sigma^8 M \ar[ur]^{v_1^4} \ar[d]_{\pinch} & & &\\ & S^9, \ar[uuurrr]^t & & &  }    \]
where $t= \mu, \mu +\eta\varepsilon$, or $\mu + \eta^2\sigma$. Since $t \circ \pinch= \eta \circ \pinch \circ v_1^4$, it folows that
\[\incl \circ t \circ \pinch= \incl \circ \eta \circ \pinch\circ v_1^4=2v_1^4. \]

Here we used that the mod-$2$ Moore spectrum satisfies $$2 Id_M = \incl\circ\eta\circ\pinch.$$

It remains to show that $\incl \circ \eta\varepsilon \circ \pinch=0$ and $\incl \circ \eta^2\sigma \circ \pinch=0$. We will now show that given $x \in \pi_8S^0_{(2)}$, then $\incl \circ \eta x \circ \pinch=0$. Indeed, since $2x=0$, using the cofiber sequence
\[
\Sigma^{-1}M \xrightarrow{\pinch}  S^0 \xrightarrow{2}  S^0 \xrightarrow{\incl} M
 \]
there exists $\widetilde{x}: S^8 \rightarrow \Sigma^{-1}M$ such that $\pinch \circ \widetilde{x} =x$. Now one can compute
\[\incl \circ \eta x \circ \pinch = \incl \circ \eta \circ \pinch \circ \widetilde{x} \circ \pinch= 2 \circ \widetilde{x} \circ \pinch\]
because $\incl \circ \eta \circ \pinch =2$. 
But $2\pinch=0$ which finishes the proof. \end{proof}


Before starting the calculations we recall the isomorphism
\begin{align*}
[\SigmaG, \SigmaG]^{G,v_1}_* \cong \mathbb{Z}[G]\otimes [L_1 S^0, L_1 S^0]^{\Sp}_*.
\end{align*}

The right hand side is generated by elements of the form $\alpha\cdot g$, where $\alpha \in \pi_*(L_1S^0)$ and $g \in G$. On the left hand side, this corresponds to the map $$\SigmaG \wedge L_1S^0 \longrightarrow \SigmaG \wedge L_1S^0$$ that is $\alpha$ on the factor $L_1S^0$ and multiplication by $g$ on the factor $\SigmaG$. We also recall that there is an isomorphism

\begin{align*}
[\SigmaG, \SigmaG]^G_*\cong \mathbb{Z}[G]\otimes [ S^0, S^0]^{\Sp}_*, 
\end{align*}
and similar observations apply also here. 

\begin{lemma}\label{lem:maincalc}
Let $F: \Ho(\SpG) \longrightarrow \Ho(\LSpG)$ be an exact functor such that 
$$F(\SigmaG)=\SigmaG \wedge L_1S^0$$ 
and $F(g)=g$ for any $g \in G$. Then
\[
F: [\SigmaG , \SigmaG]^G_* \longrightarrow [\SigmaG \wedge L_1S^0, \SigmaG \wedge L_1S^0]^{G,v_1}_*
\]
satisfies
\begin{itemize}
\item $F(\eta)= \eta + \sum\limits_{g \in B} y_1\cdot g$ for some $B \subseteq G$ 
\item $F(\eta^2)=\eta^2$ \vphantom{$ \sum\limits_{g \in G \backslash \{1\}}$}
\item $F(\nu)=m\cdot \nu + \sum\limits_{g \in G \backslash \{1\}} n_g\nu \cdot g$, $m$ odd, $n_g$ even
\item $F(\sigma)=k\sigma + \sum\limits_{g \in G\backslash\{1\}} l_g\sigma\cdot g$, $k$ odd, $l_g$ even
\item $F(\mu)=\mu+\sum\limits_{g\in C} \eta^2\sigma\cdot g$ for some $C \subseteq G$
\end{itemize}
\end{lemma}

\begin{proof} When we write e.g. $F(\eta)$ in the following calculation, we mean $F(\SigmaG \wedge \eta)$. Having $F(\SigmaG)= \SigmaG \wedge L_1S^0$ in mind, we consider $F$ to be a ring homomorphism from $\mathbb{Z}[G]\otimes \pi_*S^0$ to $\mathbb{Z}[G]\otimes \pi_*(L_1S^0)$.

\bigskip

\fbox{$F(\eta)= \eta + \sum\limits_{g \in B} y_1\cdot g$ for some $B \subseteq G$} 

Because $F$ is triangulated, we have an isomophism $F(\SigmaG \wedge M) \cong \SigmaG \wedge L_1M$ in $\Ho(\LSpG)$ such that the following diagram commutes: 
$$\xymatrix{F(\SigmaG) \ar[r]^2 \ar@{=}[d] & F(\SigmaG) \ar[rr]^-{F(1 \wedge \incl)} \ar@{=}[d] & & F(\SigmaG \wedge M) \ar[d]^-\cong \ar[rr]^-{F(1 \wedge \pinch)} & & F(\Sigma \SigmaG) \ar[d]^-{\cong}  \\ \SigmaG \wedge L_1S^0  \ar[r]^2 & \SigmaG \wedge L_1S^0 \ar[rr]^-{1 \wedge \incl} & & \SigmaG \wedge L_1M  \ar[rr]^-{1 \wedge \pinch} & & \Sigma \SigmaG \wedge L_1S^0.}$$

Having this diagram in mind, we will once and for all identify $F(\SigmaG \wedge M)$ with \linebreak $\SigmaG \wedge L_1M$, and $F(\Sigma \SigmaG)$ with $\Sigma \SigmaG \wedge L_1S^0$. Under these identifications $F(1 \wedge \incl)$ and $F(1 \wedge \pinch)$ correspond to $1 \wedge \incl$ and $1 \wedge \pinch$, respectively. Below we will also denote for short $1 \wedge \incl$ and $1 \wedge \pinch$ just by $\incl$ and $\pinch$.

The mod-$2$ Moore spectrum has the peculiar property that $$2 Id_M = \incl\circ\eta\circ\pinch \neq 0.$$ As $\eta$ survives $K$-localisation the same identity also holds $K$-locally, and we also have $2 Id_{L_1M}\neq 0$ (see \cite[Section 5.2]{Roi07}).  

We know that 
\begin{multline}
F(\eta)=\sum\limits_{g\in A} \eta\cdot g + \sum\limits_{g \in B} y_1\cdot g \nonumber\\ \in [\SigmaG, \SigmaG]^{G,v_1}_1 \cong \bigoplus\limits_{|G|}\pi_1(L_1S^0) = \mathbb{Z}/2\{ \eta \cdot g, y_1 \cdot g'\,\,|\,\,g, g' \in G\}
\end{multline}
for some subsets $A$ and $B$ of $G$. Because  $1 \wedge 2 Id_{L_1M}=F(1 \wedge 2Id_{M})\neq 0$, we have
\begin{eqnarray}
2&=&(1 \wedge \incl)\circ F(\eta) \circ (1 \wedge \pinch) \nonumber\\ &=& \sum\limits_{g\in A} (\incl\circ\eta\circ\pinch)\cdot g + \sum\limits_{g \in B}(\incl\circ y_1\circ\pinch)\cdot g = \sum\limits_{g\in A}2\cdot g \nonumber
\end{eqnarray}
because $\incl\circ y_1\circ\pinch=0$ \cite[Section 5.2]{Roi07}. Comparing coefficients this must mean that $A=\{1\}$, so
\[
F(\eta)= \eta + \sum\limits_{g \in B} y_1\cdot g.
\]

\bigskip
\fbox{$F(\eta^2)=\eta^2$} 

From the previous point we can conclude immediately that
\[
F(\eta^2)=F(\eta)^2=\eta^2 + \sum\limits_{g \in B}2y_1 \eta g + \sum\limits_{g,g' \in B}y_1^2 g g' = \eta^2
\]
because $y_1\eta=y_1^2=0 \in \pi_2(L_1S^0)$.

\bigskip
\fbox{$F(\nu)=m\cdot \nu + \sum\limits_{g \in G \backslash \{1\}} n_g\nu \cdot g$, $m$ odd, $n_g$ even}

Because $\eta^3=4\nu$ inside $\pi_3S^0$ and $\pi_3(L_1S^0)$, we can use our calculation for $F(\eta)$ together with the relation $\eta^2 y_1=0$ to see that
\[
4F(\nu)=F(\eta^3)=\eta^3=4\nu.
\]
Furthermore,
\[
F(\nu)=m\cdot \nu + \sum\limits_{g \in G \backslash \{1\}} n_g\nu \cdot g
\]
for some integers $m$ and $n_g$. Hence
\[
4\nu=4m\nu + \sum\limits_{g \in G \backslash \{1\}} (4n_g)\nu \cdot g \in \mathbb{Z}/8\{\nu\cdot g\,\,|\,\, g \in G\}.
\]
From here we can conclude that $m$ must be odd and $n_g$ even. 

\bigskip
\fbox{$F(\sigma)=k\sigma + \sum\limits_{g \in G\backslash\{1\}} l_g\sigma\cdot g$, $k$ odd, $l_g$ even} 

We know that
\[
8\sigma=\left< \nu, 8, \nu \right>
\]
with trivial indeterminacy. Thus,
\begin{eqnarray}
8F(\sigma) &= & \left< F(\nu), 8, F(\nu) \right> \nonumber \\
& = & \left< m\cdot \nu + \sum\limits_{g \in G \backslash \{1\}} n_g\nu \cdot g, \,8\, , m\cdot \nu + \sum\limits_{g \in G \backslash \{1\}} n_g\nu \cdot g \right> \nonumber \\
& = & \left< m\nu, 8, m\nu \right> + \sum\limits_{g,g'}\left< n_g\nu, 8, n_{g'}\nu \right>\cdot gg' \nonumber \\
&  & + \sum\limits_g\left< m\nu, 8, n_g\nu \right> \cdot g+ \sum\limits_g\left< n_g\nu, 8, m\nu \right>\cdot g \nonumber \\
& = & m^2\left< \nu, 8, \nu \right> \nonumber
\end{eqnarray}
because all other summands are trivial as they are multiples of 16. Thus,
\[
8F(\sigma)=8m^2\sigma, \,\,\,\,m \,\,\mbox{odd}.
\]
We also know that 
\[
F(\sigma)=k\sigma + \sum\limits_{g \in G\backslash\{1\}} l_g\sigma\cdot g
\]
for some integers $k$ and $l_g$. A similar argument to the previous point shows that $k$ is odd and the $l_g$ are even.

\bigskip
\fbox{$F(\mu)=\mu+\sum\limits_{g\in C} \eta^2\sigma\cdot g$ for some $C \subseteq G$} 

We use the Toda bracket $\mu \in \left<2, 8\sigma, \eta\right>$. Thus,
\[
F(\mu) \in \left<2, F(8\sigma), F(\eta)\right>.
\]
Inserting what we have calculated for $F(8\sigma)$ and $F(\eta)$, we see that
\[
 \left<2, F(8\sigma), F(\eta)\right>= \left< 2, 8\sigma, \eta + \sum\limits_{g \in B} y_1\cdot g \right>.
\]
This bracket is contained in
\[
\left<2, 8\sigma, \eta \right> + \sum\limits_{g \in B} \left<2, 8\sigma, y_1 \right> \cdot g.
\]
We know that 
\[
\left<2, 8\sigma, \eta \right>= \{ \mu, \mu+\eta^2\sigma \}
\]
and
\[
\left< 2, 8\sigma, y_1 \right>  = \eta^2 \sigma \,\,\,\mbox{with zero indeterminacy}
\]
in $\pi_*(L_1S^0)$. Putting these things together we can conclude that
\[
F(\mu)=\mu+\sum_{g \in C} \eta^2\sigma\cdot g \,\,\,\mbox{for some $C \subseteq G$}. 
\]
\end{proof}


\begin{corollary}\label{cor:v1}
Let $F: \Ho(\SpG) \longrightarrow \Ho(\LSpG)$ be an exact functor such that 
$$F(\SigmaG)=\SigmaG \wedge L_1S^0$$ 
and $F(g)=g$ for any $g \in G$. Then $F(\SigmaG \wedge v_1^4)$ is an isomorphism in $\Ho(\LSpG)$. 
\end{corollary}

\begin{proof}
Denote $\SigmaG \wedge v_1^4$ by $v_1^4$ for short. We know that $2\vv = \incl\circ\mu\circ\pinch$ in $\Ho(\Sp)$ (see Lemma \ref{v1relation}) and hence also in $\Ho(L_1 \Sp)$. So 
\begin{eqnarray}
2F(\vv) & = & (1 \wedge \incl)\circ F(\mu) \circ (1 \wedge \pinch) \nonumber\\
& = & \incl \circ (\mu + \sum\limits_{g \in C} \eta^2\sigma \cdot g) \circ \pinch \nonumber\\
& = & \incl \circ \mu \circ \pinch + \sum\limits_{g \in C} (\incl\circ\eta^2\sigma \circ \pinch)\cdot g \nonumber\\
& = & 2\vv \nonumber
\end{eqnarray}
because $\incl\circ\eta^2\sigma\circ\pinch=0$ as we proved in Lemma \ref{v1relation}. 

We know from \cite[Corollary 5.6]{Roi07} that $v_1^4 \in [M,M]^{L_1\Sp}_8$ has order $4$. This implies that 
\[
F(\vv)=av_1^4 + \sum\limits_{g \in D} 2 v_1^4 g + \sum\limits_{g \in E} (\tilde{\eta\sigma}\circ\pinch)\cdot g + \sum\limits_{g \in F} (id\wedge \eta\sigma) \cdot g
\]
where $a$ is odd and $D, E, F \subset G$.
The elements $\tilde{\eta\sigma} \circ \pinch$ and $id \wedge \eta\sigma$ in $[M,M]^{L_1\Sp}_8$ are described in more detail in \cite[Section 5]{Roi07} but we only need to know here that they have order 2. These order $2$ elements induce zero on ${K_{(2)}}_*(-)$ for degree reasons \cite[Section 3.2]{Roi07}, leaving us with the result that $F(\vv)$ is a $K_*$-equivalence and thus a $v_1$-local isomorphism. \end{proof}

Now we turn to the case of an exact endofunctor. 

\begin{proposition}\label{cor:maincalc}
Let $\mathbb{F}: \Ho(\LSpG) \longrightarrow \Ho(\LSpG)$ be an exact endofunctor such that $\mathbb{F}(L_{v_1}\SigmaG)=L_{v_1}\SigmaG$
and $\mathbb{F}(g)=g$ for any $g \in G$. Then 
\[\mathbb{F}: [L_{v_1}\SigmaG, L_{v_1}\SigmaG]^{G,v_1}_n \rightarrow [L_{v_1}\SigmaG, L_{v_1}\SigmaG]^{G,v_1}_n\] 
is an isomorphism for $n=0,...,8$.
\end{proposition}


\begin{proof} This is a calculation very similar to the one in the proof of Lemma \ref{lem:maincalc}. Using exactly the same arguments as there we see that the following hold:

\begin{itemize}
\item $\mathbb{F}(\eta)= \eta + \sum\limits_{g \in B} y_1\cdot g$ for some $B \subseteq G$ 
\item $\mathbb{F}(\eta^2)=\eta^2$ \vphantom{$ \sum\limits_{g \in G \backslash \{1\}}$}
\item $\mathbb{F}(\nu)=m\cdot \nu + \sum\limits_{g \in G \backslash \{1\}} n_g\nu \cdot g$, $m$ odd, $n_g$ even
\item $\mathbb{F}(\sigma)=k\sigma + \sum\limits_{g \in G\backslash\{1\}} l_g\sigma\cdot g$, $k$ odd, $l_g$ even
\item $\mathbb{F}(\mu)=\mu+\sum\limits_{g\in C} \eta^2\sigma\cdot g$ for some $C \subseteq G$
\end{itemize}

We first show that $\mathbb{F}$ induces isomorphisms in degrees $n=2,...,8$. After this we will take care of $y_0$ and $y_1$. 

Since $\mathbb{F}(\eta^2)=\eta^2$, the case $n=2$ immediately follows. Next, the formula for $\mathbb{F}(\nu)$ implies that 
\[
\mathbb{F}: [L_{v_1}\SigmaG, L_{v_1}\SigmaG]^{G,v_1}_3 \cong \bigoplus\limits_{|G|} \mathbb{Z}/8 \longrightarrow \bigoplus\limits_{|G|}\mathbb{Z}/8 \cong [L_{v_1}\SigmaG, L_{v_1}\SigmaG]^{G,v_1}_3
\]
is represented by a matrix with odd determinant, i.e. an invertible matrix. This gives that 
\[
\mathbb{F}: [L_{v_1}\SigmaG, L_{v_1}\SigmaG]^{G,v_1}_3 \longrightarrow [L_{v_1}\SigmaG, L_{v_1}\SigmaG]^{G,v_1}_3
\]
is an isomorphism.

Similarly the formula for $\mathbb{F}(\sigma)$ tells us that
\[
\mathbb{F}: [L_{v_1}\SigmaG, L_{v_1}\SigmaG]^{G,v_1}_7 \longrightarrow [L_{v_1}\SigmaG, L_{v_1}\SigmaG]^{G,v_1}_7
\]
is an isomorphism.

Furthermore, since $2\eta=0$ and $\sigma y_1=0$, it follows from the formulas for $\mathbb{F}(\eta)$ and $\mathbb{F}(\sigma)$ that $\mathbb{F}(\eta \sigma)= \eta \sigma$. Hence we conclude that $\mathbb{F}$ gives an isomophism in degree 8. 


Now we take care of dimensions $n=0$ and $n=1$. 

\bigskip
\fbox{$\mathbb{F}(y_0)=y_0$}

As $y_0$ is a torsion element, $\mathbb{F}(y_0)$ must also be torsion, so
\[
\mathbb{F}(y_0)=\sum\limits_{g \in D} y_0\cdot g.
\]
We would like to know what $D$ is. We make use of the relation $\mu y_0=\eta^2\sigma$, so $$\mathbb{F}(\mu)\mathbb{F}(y_0)=\mathbb{F}(\eta^2)\mathbb{F}(\sigma)=\eta^2\sigma.$$ Substituting the results of our previous calculations, the left hand side is
\begin{eqnarray}
\mathbb{F}(\mu)\mathbb{F}(y_0) &=& (\mu+\sum_{g \in C} \eta^2\sigma\cdot g)(\sum\limits_{g \in D} y_0\cdot g) \nonumber \\
& = & \sum\limits_{g \in D} \mu y_0 \cdot g + \sum\limits_{g \in C, g' \in D} (\eta^2\sigma y_0)\cdot gg'. \nonumber
\end{eqnarray}
Because $\eta^2\sigma y_0 = \mu y_0^2=0$, the last summand is zero. Hence we arrive at 
\[
\sum\limits_{g \in D} \mu y_0 g=\eta^2\sigma=\mu y_0,
\]
so $D=\{1\}$ and $\mathbb{F}(y_0)=y_0$. 

\bigskip
\fbox{$\mathbb{F}(y_1)=y_1$}

The last calculation is now simple as $y_1=\eta y_0$. Thus,
\[
\mathbb{F}(y_1)=\mathbb{F}(\eta)\mathbb{F}(y_0)= (\eta + \sum\limits_{g \in B}y_1 \cdot g)y_0=\eta y_0 + \sum\limits_{g \in B} y_1y_0 \cdot g = y_1
\]
as $y_1 y_0=0$ and $\eta y_0=y_1$.

\bigskip
We finally note that because $\mathbb{F}(g)=g$, the element $y_1\cdot g$ is in the image of $\mathbb{F}$ for any $g$. So $\eta$ is also in the image of $\mathbb{F}$, which makes $\mathbb{F}$ surjective and consequently bijective in degree 1. \end{proof}

\begin{corollary} \label{moorefree} Let $\mathbb{F}: \Ho(\LSpG) \longrightarrow \Ho(\LSpG)$ be an exact endofunctor such that $\mathbb{F}(L_{v_1}\SigmaG)=L_{v_1}\SigmaG$
and $\mathbb{F}(g)=g$ for any $g \in G$. Then the map
\[\mathbb{F}: [L_{v_1}\SigmaG \wedge M, L_{v_1}\SigmaG]^{G,v_1}_* \longrightarrow [\mathbb{F}(L_{v_1}\SigmaG \wedge M), \mathbb{F}(L_{v_1}\SigmaG)]^{G,v_1}_* \]
is an isomorphism. \end{corollary}

\begin{proof} Recall we have a cofiber sequence
\[\xymatrix{L_{v_1}\SigmaG \ar[r]^2 & L_{v_1}\SigmaG \ar[r] & L_{v_1}\SigmaG \wedge M \ar[r] & \Sigma L_{v_1}\SigmaG. }\]
Using this, the previous proposition and the Five Lemma, one can see that the map under consideration is an isomorphism in degrees $*=0,...,7$. The other degrees follow from the $8$-periodicity of $\vv$, i.e. the fact that the morphism 
\[L_{v_1}(1 \wedge v_1^4):  L_{v_1}\SigmaG \wedge \Sigma^8 M \longrightarrow L_{v_1}\SigmaG \wedge M\] 
is an isomorphism in $\Ho(\LSpG)$.   
\end{proof}

\section{The Quillen functor (revisited)}\label{sec:revisited}

This entire section is dedicated to the proof of Proposition \ref{prop:mainprop}, namely that for $X$ cofibrant and $X\cong \Psi(L_{v_1}S^0_G)$, the $G$-spectrum $\mathbf{R}\Hom(X,Y)$ is $v_1$-local for all $Y$. This will be an inductive argument via the subgroups $H$ of $G$. 




Let $F$ denote the composite
\[ F: \Ho(\SpG) \xrightarrow{- \wedge X} \Ho(\C) \xrightarrow{\Psi^{-1}} \Ho(\LSpG).  \]

We will first explain how showing that $$F(\SigmaGH \wedge \vv)$$ is an isomorphism for all subgroups $H$ of $G$ implies the claim. Recall that we started with an equivalence
\[
\Psi: \Ho(\LSpG) \longrightarrow \Ho(\mathcal{C})
\]
such that $\Psi(L_{v_1}\SigmaGH) \cong    \SigmaGH \wedge^{\mathbf{L}} \Psi(L_{v_1} S^0_G)$, and that these isomorphisms are natural with respect to transfers, restrictions and conjugations. Let $X$ denote a cofibrant replacement of $\Psi(L_{v_1} S^0_G)$ in $\C$. As described earlier, we have to show that the adjunction
\[
- \wedge X: \SpG \lradjunction \mathcal{C}: \Hom(X,-)
\]
factors through $\LSpG$. This will be the case (by Lemma \ref{lem:factor}) if we manage to show that for any fibrant $Y \in \C$ the $G$-spectrum
$\Hom(X,Y)$
is $v_1$-local. By definition, the latter is the case if and only if for any $H \leq G$ the map
\[
[\SigmaGH \wedge M, \Hom(X,Y)]^G_* \xrightarrow{(\SigmaGH \wedge \vv)^*} [\SigmaGH\wedge M, \Hom(X,Y)]^G_{*+8}
\]
is an isomorphism. Hence it suffices to show that the map
\[\SigmaGH \wedge \Sigma^8 M \wedge X \xrightarrow{\SigmaGH \wedge v_1^4 \wedge 1} \SigmaGH \wedge M \wedge X  \]
is an isomorphism in $\Ho(\C)$ for any $H \leq G$. But the latter holds if and only if the composite 
\[ F: \Ho(\SpG) \xrightarrow{- \wedge X} \Ho(\C) \xrightarrow{\Psi^{-1}} \Ho(\LSpG)  \]
sends $\SigmaGH \wedge \vv$ to an isomorphism. We will show this by induction on $H$.

For convenience we choose the inverse $\Psi^{-1}$ so that 
$$\Psi^{-1}(\SigmaGH \wedge X)= L_{v_1}\SigmaGH.$$ 
By our assumptions on $\Psi$, we see that $F$ satisfies the conditions of Lemma \ref{lem:maincalc}. 

We showed the case $H=\{1\}$ in Section \ref{sec:calculation}, in other words we have already proved that $F(\SigmaG \wedge \vv)$ is an isomorphism. This will serve as the basis for an induction on the size of the subgroup $H$. We will proceed to prove the claim inductively for nontrivial subgroups using isotropy separation. 

Assume that $H \leq G$ and for all proper subgroups $K$ of $H$ the morphism $$F(\SigmaGK \wedge \vv)$$ is an isomorphism. We want to show that this assumption implies that $$F(\SigmaGH \wedge \vv)$$ is an isomorphism. Let $\PP(H)$ denote the family of proper subgroups in $H$ and $\PP(H|G)$ denote the family of subgroups in $G$ which are proper subconjugates of $H$. Consider the full subcategory of $\Ho(\SpG)$ of those objects $Z$ such that the map
\[F(1 \wedge \vv) \colon F(Z \wedge \Sigma^8 M) \rightarrow F(Z \wedge M)  \]
is an isomorphism. Since $F$ is exact and preserves infinite coproducts, this subcategory is localising, and hence the morphism
\[F(\GEPH \wedge \Sigma^8 M) \xrightarrow{F(1 \wedge \vv)} F(G \ltimes_H {\Sigma^\infty E\PP(H)}_{+} \wedge M)   \]
is an equivalence. 
To complete the induction step we need to show that
\[ F(\GEPH \wedge M) \xrightarrow{F(\proj)} F(\SigmaGH \wedge M) \]
is a $\PP(H|G)$-equivalence. 
This follows from the next lemma:

\begin{lemma} \label{Fequipres}Let $f: X \to Y$ be a cellular map of $G$-CW complexes that is a stable $\FF$-equivalence for some family $\FF$. Then the map
\[F(\Sigma^{\infty}f_{+}): F(\Sigma^{\infty}X_{+}) \rightarrow F(\Sigma^{\infty}Y_{+})  \]
is an $\FF$-equivalence.
 
\end{lemma}

\begin{proof} The key idea is that $F$ preserves cellular chains.   Let $\mathbb{K}$ denote the inflation $\varepsilon^*K$ of the $2$-local equivariant $K$-theory spectrum $K$. Then a map $A \rightarrow B$ in $\Ho(\LSpG)$ is an $\FF$-equivalence if and only if $A \wedge \mathbb{K} \rightarrow B \wedge \mathbb{K}$ is an $\FF$-equivalence in $\Ho(\SpG)$.

Consider any $U \in \FF$. Since $F$ commutes with infinite coproducts, the cellular filtrations of $X$ and $Y$ give us strongly convergent Atiyah-Hirzebruch type spectral sequences
\[E^1_{pq}=\pi^U_q(F(\Sigma^{\infty} {X^p/X^{p-1}}_+) \wedge \mathbb{K}) ) \Rightarrow \pi^U_{p+q}(F(\Sigma^{\infty} X_{+}) \wedge \mathbb{K}) \]
and
\[E^1_{pq}=\pi^U_q(F(\Sigma^{\infty} {Y^p/Y^{p-1}}_+) \wedge \mathbb{K})) \Rightarrow \pi^U_{p+q}(F(\Sigma^{\infty} Y_{+}) \wedge \mathbb{K}). \]
Since the map $f$ is cellular, it induces a map of the latter spectral sequences. As furthermore $F(\SigmaGH) = L_{v_1}\SigmaGH$ and $F$ respects conjugations, restrictions and transfers, we can see that there is a commutative diagram for $X$ (we omit the $\Sigma^{\infty}$ symbol from notation below)
\[\xymatrix{\cdots \ar[r] & F(\Sigma^{-p}{X^p/X^{p-1}}_+) \ar[d]^\cong \ar[r]^-{F(\partial_p)} & F(\Sigma^{-p+1}{X^{p-1}/X^{p-2}}_+) \ar[d]^\cong \ar[r] & \cdots \ar[r] & F(X^0_+) \ar[d]^\cong \\  \cdots \ar[r] & L_{v_1}(\Sigma^{-p}{X^p/X^{p-1}}_+) \ar[r]^-{L_{v_1}(\partial_p)} & L_{v_1}(\Sigma^{-p+1}{X^{p-1}/X^{p-2}}_+) \ar[r] & \cdots \ar[r] & L_{v_1}(X^0_+) }\]
and similarly for $Y$. For the same reason, we also have a commutative diagram
\[\xymatrix{F(\Sigma^{-p}{X^p/X^{p-1}}_+) \ar[d]^\cong \ar[r]^-{F(f)} & F(\Sigma^{-p}{Y^p/Y^{p-1}}_+) \ar[d]^{\cong} \\ L_{v_1}(\Sigma^{-p}{X^p/X^{p-1}}_+) \ar[r]^-{L_{v_1}(f)} & L_{v_1}(\Sigma^{-p}{Y^p/Y^{p-1}}_+),  } \]
where the isomorphisms are the same as in the previous diagram. Roughly speaking, these facts imply that $F$ preserves cellular chain complexes. Now this allows us to conclude that $f$ induces an isomorphism on $E^2$ terms. Indeed, the first commutative diagram implies that there are isomorphisms
\[E^2_{pq}(X) \cong H_p^U(X, \pi_q^{-} \mathbb{K}) \]
and
\[E^2_{pq}(Y) \cong H_p^U(Y, \pi_q^{-} \mathbb{K}), \]
where $H_p^U(-, \pi_q^{-} \mathbb{K})$ denotes Bredon homology. On the other hand the second commutative diagram gives us a commutative diagram
%
\[\xymatrix{E^2_{pq}(X) \ar[d]^\cong \ar[r]^{F(f)_*} & E^2_{pq}(Y) \ar[d]^\cong \\ H_p^U(X, \pi_q^{-} \mathbb{K}) \ar[r]^{f_*} & H_p^U(Y, \pi_q^{-} \mathbb{K}). } \]
But the lower horizontal map is an isomoprhism since $U$ is in $\FF$ and $f$ is a stable $\FF$-equivalence. This finishes the proof. \end{proof}

Now we go back to our induction to show that $F(\SigmaGH \wedge \vv)$ is an isomorphism. Lemma \ref{Fequipres} implies that the morphism $$F(\proj \wedge 1): F(\GEPH \wedge M) \rightarrow F(\SigmaGH \wedge M)$$ is a $\PP(H|G)$-equivalence. Next, the commutative diagram
\[\xymatrix{F(\GEPH \wedge \Sigma^8 M) \ar[r]^-{F(1 \wedge \vv)} \ar[d]^{F(\proj \wedge 1)} & F(\GEPH \wedge M) \ar[d]^{F(\proj \wedge 1)} \\ F(\SigmaGH \wedge \Sigma^8 M) \ar[r]^-{F(1 \wedge \vv)}  &  F(\SigmaGH \wedge M) } \]
shows that the map 
\[F(1 \wedge \vv): F(\SigmaGH \wedge \Sigma^8 M) \rightarrow F(\SigmaGH \wedge M) \]
is a $\PP(H|G)$-equivalence. In order to show that it is an isomorphism, it suffices to show that the induced map on $H$-geometric fixed points
\[\Phi^H(F(1 \wedge \vv)): \Phi^H(F(\SigmaGH \wedge \Sigma^8 M)) \rightarrow \Phi^H(F(\SigmaGH \wedge M)) \]
is an isomorphism. For this consider the composite ($N_G(H)$ denotes the normalizer of $H$ in $G$)
\begin{multline} 
\Ho(\Sp_{W_G(H)}) \xrightarrow{\varepsilon^*}  \Ho(\Sp_{N_G(H)}) \xrightarrow{G \ltimes_{N_G(H)} -}  \Ho(\Sp_G) \xrightarrow{F} \Ho(\LSpG) \rightarrow  \nonumber\\\xrightarrow{\Res_{N_G(H)}^G}  \Ho(L_{v_1}\Sp_{N_G(H)}) \xrightarrow{\Phi^H}  \Ho(L_{v_1}\Sp_{W_G(H)}). 
 \end{multline}
which we denote by $$\hat{F} \colon \Ho(\Sp_{W_G(H)}) \rightarrow \Ho(L_{v_1}\Sp_{W_G(H)}).$$ When restricted to free $W_G(H)$-spectra, $\hat{F}$ satisfies the conditions of Lemma \ref{lem:maincalc}, implying that $$\hat{F}(\Sigma^{\infty}W_G(H)_+ \wedge \vv): \hat{F}(\Sigma^{\infty} W_G(H)_+ \wedge \Sigma^8M) \longrightarrow \hat{F}(\Sigma^{\infty} W_G(H)_+ \wedge M) $$ is an isomorphism. But by unraveling definitions we see that the latter map is equal to $\Phi^H(F(\SigmaGH \wedge \vv ))$ which finishes the proof of the fact that the spectrum $\mathbf{R}\Hom(X,Y)$ is $v_1$-local.

We conclude this section by the following corollary of Lemma \ref{Fequipres} which will be used below:

\begin{corollary}\label{lem: F(GEPH)} There is an isomorphism $$F(\GEPH ) \cong L_{v_1} \GEPH$$ such that the diagram
\[\xymatrix{F(\GEPH) \ar[r]^{\cong} \ar[d]^{F(\proj)} & L_{v_1} \GEPH \ar[d]^{\proj} \\ F(\SigmaGH) \ar@{=}[r] & L_ {v_1} \SigmaGH   } \]
commutes.
\end{corollary}

\begin{proof} The cellular map 
\[ \proj \colon \GEPH \rightarrow \SigmaGH\] is a $\PP[H|G]$-equivalence, i.e. a $\pi^K_*$-isomorphism for all subgroups $K$ which are subconjugated to $H$. Hence so is its localisation 
\[\proj \colon L_{v_1}\GEPH \rightarrow L_{v_1}\SigmaGH.\]
By \cite[Lemma 7.1.5]{P16}  it suffices to show that $F(\proj)$ is a $\PP[H|G]$-equivalence. This now follows from Lemma \ref{Fequipres}. \end{proof}


\section{The Quillen equivalence}\label{sec:quillenequivalence}

We showed in Section \ref{sec:revisited} that the Quillen adjunction from Section \ref{sec:quillenfunctor}
\[
- \wedge X: \SpG \lradjunction \mathcal{C}: \Hom(X,-)
\]
factors through $\LSpG$, i.e. that
\[
- \wedge X: \LSpG \lradjunction \mathcal{C}: \Hom(X,-)
\]
is also a Quillen adjunction. Moreover the diagram of left Quillen functors
\[\xymatrix{\SpG \ar[rr]^-{- \wedge X} \ar[d]^-{1} & & \mathcal{C} \\ \LSpG \ar[urr]_-{- \wedge X} & &  }\]
commutes on the nose. In this section we show that our adjunction is actually a Quillen equivalence.  

\bigskip
Define the exact functor $\mathbb{F} : \Ho(\LSpG) \rightarrow \Ho(\LSpG)$ to be the composite
\[\Ho(\LSpG) \xrightarrow{- \wedge^{\mathbf{L}} X } \Ho(\C) \xrightarrow{\Psi^{-1}} \Ho(\LSpG).\]
By the commutativity of the latter diagram we see that the diagram
\[\xymatrix{\Ho(\SpG) \ar[r]^-{F} \ar[d]^{L_{v_1}} & \Ho(\LSpG) \\ \Ho(\LSpG) \ar[ur]_-{\FFF} &   }\]
commutes up to a natural isomorphism. Our assumptions on $\Psi$ imply that we have 
\[\FFF(L_{v_1} \SigmaGH ) = L_{v_1} \SigmaGH   \]
for any subgroup $H$ of $G$ and $\mathbb{F}$ respects conjugations, restrictions and transfers. 

\begin{lemma} \label{lem: FFF(M)} There is a commutative diagram

\[\xymatrix{\FFF(L_{v_1}\GEPH  ) \ar[d]^{\FFF(\proj )} \ar[r]^-{\cong} & L_{v_1}\GEPH  \ar[d]^{\proj } \\ \FFF(L_{v_1}\SigmaGH )  \ar@{=}[r] & L_{v_1}\SigmaGH } \]

with the upper horizontal arrow being an isomorphism.

\end{lemma}

\begin{proof}
This follows from Corollary \ref{lem: F(GEPH)} and the commutative diagram 
\[\xymatrix{\FFF(L_{v_1}\GEPH ) \ar[d]^{\FFF(\proj)} \ar[r]^{\cong} & F(\GEPH) \ar[d]^{F(\proj)} &  \\ \FFF(L_{v_1}\SigmaGH )  \ar@{=}[r] & F(\SigmaGH).} \]
\end{proof}

\subsection{Reduction to mod-$2$} \label{mod2} In order to prove the main theorem, it suffices to show that the functor $\mathbb{F}$ is an equivalence which amounts to showing that the map
\[ \FFF: [L_{v_1}\SigmaGH, L_{v_1}\SigmaGK]^{G,v_1}_* \longrightarrow [\FFF(L_{v_1}\SigmaGH), \FFF(L_{v_1}\SigmaGK)]^{G,v_1}_*\]
is an isomorphism for any $H,K \leq G$. (As the set of objcets $\{L_{v_1}\SigmaGH \; | \; H \leq G  \}$ compactly generates $\Ho(\LSpG)$ (Proposition \ref{generators}).)

The techniques from \cite{P16} cannot be directly applied here since the source and target are not finite when $\ast=-2$ and $\ast=0$. However after mod-$2$ reduction they become finite, which will be proved below. Before that we will show that the integral result follows from the combination of the rational result with the mod-$2$ result. 

The map 
\[ \FFF: [L_{v_1}\SigmaGH, L_{v_1}\SigmaGK]^{G,v_1}_* \longrightarrow [\FFF(L_{v_1}\SigmaGH), \FFF(L_{v_1}\SigmaGK)]^{G,v_1}_*\]
is isomorphic to the map induced on homotopy groups by the morphism of mapping spectra (here $fib$ denotes a fibrant replacement)

\[
\xymatrix{ \Hom_{\LSpG}(L_{v_1}\SigmaGH, L_{v_1}\SigmaGK) \ar[d]_{- \wedge X} \\  \Hom_{\mathcal{C}}(L_{v_1}\SigmaGH \wedge X, L_{v_1}\SigmaGK \wedge X) \ar[d] \\  \Hom_{\mathcal{C}}(L_{v_1}\SigmaGH \wedge X, (L_{v_1}\SigmaGK \wedge X)^{fib}). } 
\]

Since the source and target of this morphism of spectra are $2$-local, it suffices to check that the morphism is a rational equivalence and induces isomorphisms on mod-$2$ homotopy groups. We know that the functor $\FFF$ preserve transfers, conjugations and restrictions between cosets. This implies that the latter map is a rational equivalence. 

Checking that the latter map induces isomorphism on mod-$2$ homotopy groups amounts to checking that the map
\[[L_{v_1} \SigmaGH \wedge M , L_{v_1} \SigmaGK]^{G,v_1}_* \xrightarrow{-\wedge X } [L_{v_1} \SigmaGH \wedge M \wedge X , L_{v_1} \SigmaGK \wedge X]^{\C}_* \]
is an isomorphism. But this is the case if and only if the map
\[ \FFF: [L_{v_1}\SigmaGH \wedge M, L_{v_1}\SigmaGK]^{G,v_1}_* \longrightarrow [\FFF(L_{v_1}\SigmaGH \wedge M), \FFF(L_{v_1}\SigmaGK)]^{G,v_1}_*\] 
is an isomorphism.

Thus in order to prove the main theorem, it suffices to show that the last map is an isomorphism.
Since $$\FFF(L_{v_1}\SigmaGH \wedge M) \cong L_{v_1}\SigmaGH \wedge M,$$ the source and target are abstractly isomorphic. In fact the following lemma shows that they are finite.

\begin{lemma} \label{lem:finiteness} For any integer $l \in \Z$, the group
\[[L_{v_1}\SigmaGH \wedge M, L_{v_1}\SigmaGK]^{G,v_1}_l\]
is finite
 
\end{lemma}

\begin{proof} By duality we have an isomorphism
\[[L_{v_1}\SigmaGH \wedge M, L_{v_1}\SigmaGK]^{G,v_1}_l \cong [L_{v_1}\SigmaGH, L_{v_1}\SigmaGK \wedge M]^{G,v_1}_l \]

Now Corollary \ref{fixedpt} and the localised double coset formula, Wirthm\"uller isomorphism and tom Dieck splitting (see Section \ref{sec:localcat}) give us an isomorphism
\begin{multline}
[L_{v_1}\SigmaGH, L_{v_1}\SigmaGK \wedge M]^{G,v_1}_l \nonumber\\ \cong \bigoplus_{[g] \in H \setminus G /K} \;\;\; \bigoplus_{(U) \leq H \cap {}^g K } \pi_l(L_1 \Sigma^{\infty} BW_{H \cap {}^g K}(U)_+ \wedge M). 
\end{multline}

Hence it suffices to show that for any finite group $\Gamma$, the homotopy group
\[\pi_l(L_1 \Sigma^{\infty} B\Gamma_+ \wedge M) \]
is finite. The results of Adams-Baird and Ravenel (see e.g. \cite[Corollary 4.4]{Bou79}) imply that this will follow if we show that
\[KO_l(B\Gamma_+ \wedge M) \]
is finite for any $l \in \Z$. Recall that $$KO \wedge M \wedge \Cone(\eta) \simeq K(1),$$ where $K(1)$ is the mod-$2$ complex $K$-theory, i.e. the first Morava $K$-theory at prime 2. Ravenel has shown in \cite{Rav82} that $K(1)_l(B\Gamma_+)$ is finite for any integer $l$. This together with $\eta^3=0$ in $\pi_*KO$, implies that there is a finite filtration of $KO_l(B\Gamma_+ \wedge M)$ with finite  subquotients. Hence $KO_l(B\Gamma_+ \wedge M)$ is finite which proves our claim.
\end{proof}

\subsection{Reduction to the case $H=K$} \label{Reduction}

After the work of the previous subsection we can now follow \cite[Section 5]{P16} and state certain results which immediately follow from the methods of \cite{P16}. We do not provide the explicit proofs as they are verbatim translations of the proofs contained in \cite[Section 5]{P16}.

Let $H$ and $K$ be subgroups of $G$. For the rest of this subsection we fix once and for all a set $\{g \}$ of double coset representatives for $K \setminus G/ H$. Recall that for any $g \in G$, the conjugated subgroup $gHg^{-1}$ is denoted by ${}^gH$. Further let
\[
\xymatrix{
 [ L_{v_1} \SigmaGH \wedge M,  L_{v_1} \SigmaGK]_*^{G,v_1} \ar[d]_{\kappa_g} \\ [L_{v_1}  \Sigma^{\infty} G/ ({}^g H \cap K)_{+} \wedge M, L_{v_1} \Sigma^{\infty} G/({}^g H \cap K)_{+} ]_*^{G,v_1}
 }
\]
stand for the map which is defined by the following commutative diagram
$$\xymatrix{[L_{v_1} \SigmaGH \wedge M, L_{v_1} \SigmaGK]_*^{G,v_1}  \ar[d]_-{{(g \wedge 1)}^*} \ar[r]^-{\kappa_g} & [L_{v_1} \Sigma^{\infty} G/({}^g H \cap K)_{+} \wedge M, L_{v_1} \Sigma^{\infty} G/({}^g H \cap K)_{+}]_*^{G,v_1} \\ [L_{v_1} \Sigma^{\infty} G/{}^g H_{+} \wedge M, L_{v_1} \Sigma^{\infty} G/ K_{+}]_*^{G,v_1} \ar[r]^-{(\tr^K_{{}^g H \cap K})_*} & [L_{v_1} \Sigma^{\infty} G/{}^g H_{+} \wedge M, L_{v_1} \Sigma^{\infty} G/({}^g H \cap K)_{+}]_*^{G,v_1}. \ar[u]_-{(\res^{{}^g H}_{{}^g H \cap K} \wedge 1)^*}. }$$
The map
\begin{footnotesize}
\[
\xymatrix{[L_{v_1} \SigmaGH \wedge M, L_{v_1} \SigmaGK]_*^{G,v_1} \ar[d]^-{(\kappa_g)_{[g] \in K \setminus G /H}} \\ \bigoplus_{[g] \in K \setminus G /H} [L_{v_1} \Sigma^{\infty} G/({}^g H \cap K)_{+} \wedge M , L_{v_1} \Sigma^{\infty} G/({}^g H \cap K)_{+}]_*^{G,v_1}}
\]
\end{footnotesize}
is a split monomorphism. This follows exactly as \cite[Proposition 5.1.1]{P16}. Now consider the commutative diagram (keeping $\FFF(L_{v_1} \SigmaGL \wedge M) \cong L_{v_1} \SigmaGL \wedge M$ for any $L \leq G$ in mind)
\begin{tiny}\[ \xymatrix{[L_{v_1} \SigmaGH \wedge M, L_{v_1} \SigmaGK]_*^{G,v_1} \ar[d]^{\FFF} \ar[rrr]^-{(\kappa_g)_{[g] \in K \setminus G /H}} &&& \bigoplus_{[g] \in K \setminus G /H} [L_{v_1} \Sigma^{\infty} G/({}^g H \cap K)_{+} \wedge M , L_{v_1} \Sigma^{\infty} G/({}^g H \cap K)_{+}]_*^{G,v_1} \ar[d]^\FFF \\ [L_{v_1} \SigmaGH \wedge M, L_{v_1} \SigmaGK]_*^{G,v_1} \ar[rrr]^-{(\kappa_g^\FFF)_{[g] \in K \setminus G /H}} &&& \bigoplus_{[g] \in K \setminus G /H} [L_{v_1} \Sigma^{\infty} G/({}^g H \cap K)_{+} \wedge M , L_{v_1} \Sigma^{\infty} G/({}^g H \cap K)_{+}]_*^{G,v_1}.  } \]\end{tiny}
Here $\kappa_g^\FFF$ is defined as $\kappa_g$ after applying $\FFF$ to the appropriate morphisms. Since the upper horizontal map is injection, it follows that injectivity of the right map implies injectivity of the left map. By Lemma \ref{lem:finiteness}, the source and target of the left hand map are finite of the same order. The same holds for the right hand map. Hence in order to prove the main theorem it suffices to restrict our attention to the mod-$2$ endomorphisms and check that the map
\[ \FFF: [L_{v_1}\SigmaGH \wedge M, L_{v_1}\SigmaGH]^{G,v_1}_* \longrightarrow [\FFF(L_{v_1}\SigmaGH \wedge M), \FFF(L_{v_1}\SigmaGH)]^{G,v_1}_*\] 
is an isomorphism for any subgroup $H$ of $G$. 

\subsection{Proof of the main theorem} Before completing the proof of the main theorem, we state the following proposition which is an analog of Proposition 6.3.2 in \cite{P16}. The proof follows from localising the classical isotropy separation sequence and is completely analogous to the one in \cite{P16}. We do not repeat the arguments here.

\begin{proposition} \label{prop: isosep} There is a short exact sequence 
\[
\xymatrix{ 0 \ar[d] \\ [L_{v_1} G/H_+ \wedge M,  L_{v_1} G \ltimes_H {E\PP(H)}_{+}]^{G,v_1}_* \ar[d]_{\proj_*} \\ [L_{v_1} G/H_+ \wedge M,  L_{v_1} G/H_+ ]^{G,v_1}_* \ar[d]_{\Phi^H} \\ [L_{v_1} W_G(H)_+ \wedge M, L_{v_1} W_G(H)_+]^{W_G(H), v_1}_* \ar[d] \\ 0. } 
 \]
\end{proposition}
\qed

Now we obtain the last ingredients of our main proof analogously to \cite{P16}. We follow the strategy of \cite[Subsection 7.3]{P16}. Recall that what we still need to show is that
\[ \FFF: [L_{v_1}\SigmaGH \wedge M, L_{v_1}\SigmaGH]^{G,v_1}_* \longrightarrow [\FFF(L_{v_1}\SigmaGH \wedge M), \FFF(L_{v_1}\SigmaGH)]^{G,v_1}_*\] 
is an isomorphism for any subgroup $H$ of $G$. We perform an induction based on the size of $H$. The case $H=\{1\}$ was done in Section \ref{sec:calculation}. Consider the following composite

\begin{multline}
\hat{\FFF}: \Ho( L_{v_1} \Sp_{W_G(H)}) \xrightarrow{\varepsilon^*} \Ho(L_{v_1} \Sp_{N_G(H)}) \xrightarrow{G \ltimes_{N_G(H)} -}  \Ho( L_{v_1}\Sp_G) \xrightarrow{} \nonumber\\ \xrightarrow{\FFF} \Ho(\LSpG)  \xrightarrow{\Res_{N_G(H)}^G} \Ho(L_{v_1}\Sp_{N_G(H)}) \xrightarrow{\Phi^H}  \Ho(L_{v_1}\Sp_{W_G(H)}).  
\end{multline}

This $\hat{\FFF}$ satisfies the conditions of Section \ref{sec:calculation} when restricted to free $W_G(H)$-spectra. Hence the map (we omit the $\Sigma^\infty$ symbol)
\[\hat{\FFF}: [L_{v_1} W_G(H)_+ \wedge M, L_{v_1} W_G(H)_+]^{W_G(H), v_1}_* \rightarrow  [\hat{\FFF}(L_{v_1} W_G(H)_+ \wedge M), \hat{\FFF}(L_{v_1} W_G(H)_+)]^{W_G(H), v_1}_* \] is an isomorphism. Finally, we look at the following important commutative diagram

\begin{tiny}\[\hspace{-2.cm}\xymatrix{[L_{v_1}G/H_+ \wedge M, L_{v_1} G \ltimes_H E\PP(H)_+]_*^{G, v_1} \ar[r]^-{\proj_*} \ar[d]^-\FFF_{\cong} & [L_{v_1} G/H_+ \wedge M,  L_{v_1} G/H_+ ]^{G,v_1}_* \ar[d]^\FFF & [L_{v_1} W_G(H)_+ \wedge M, L_{v_1} W_G(H)_+]^{W_G(H), v_1}_* \ar[l]_-{G\ltimes_{N_G(H)} \varepsilon^*} \ar[d]^-{\hat{\FFF}}_-{\cong}\\  [\FFF(L_{v_1}G/H_+ \wedge M), \FFF(L_{v_1} G \ltimes_H E\PP(H)_+)]_*^{G, v_1} \ar[r]^-{\FFF(\proj)_*} \ar[d]_-{\cong} & [\FFF(L_{v_1}G/H_+ \wedge M), \FFF(L_{v_1}G/H_+)]^{G,v_1}_* \ar[r]^-{\Phi^H} \ar[d]^-\cong & [\hat{\FFF}(L_{v_1} W_G(H)_+ \wedge M), \hat{\FFF}(L_{v_1} W_G(H)_+)]^{W_G(H), v_1}_* \ar[d]^-\cong \\ [L_{v_1}G/H_+ \wedge M, L_{v_1} G \ltimes_H E\PP(H)_+]_*^{G, v_1} \ar[r]^-{\proj_*} & [L_{v_1} G/H_+ \wedge M,  L_{v_1} G/H_+ ]^{G,v_1}_* \ar[r]^-{\Phi^H} & [L_{v_1} W_G(H)_+ \wedge M, L_{v_1} W_G(H)_+]^{W_G(H), v_1}_*.} \]\end{tiny}

Corollary \ref{lem: F(GEPH)} implies that the lower left square commutes and that the lower left vertical map is an isomorphism. Other squares commute by definition. Further, according to Proposition \ref{prop: isosep}, the lower row in this diagram is a short exact sequence and hence so is the middle one. We also note that the upper left vertical map is an isomorphism by the induction assumption and the results of Section \ref{Reduction}. Now a diagram chase as in Subsection 7.3 of \cite{P16} shows that the map
\[ \FFF: [L_{v_1}\SigmaGH \wedge M, L_{v_1}\SigmaGH]^{G,v_1}_* \longrightarrow [\FFF(L_{v_1}\SigmaGH \wedge M), \FFF(L_{v_1}\SigmaGH)]^{G,v_1}_*\] 
is surjective. Since the source and target of this map are finite of the same cardinality, this completes the proof that the map $\FFF$ above is an isomorphism. By the results of Subsections \ref{mod2} and \ref{Reduction}, this implies that
\[ \FFF: [L_{v_1}\SigmaGH, L_{v_1}\SigmaGK]^{G,v_1}_* \longrightarrow [\FFF(L_{v_1}\SigmaGH), \FFF(L_{v_1}\SigmaGK)]^{G,v_1}_*\]
is also an isomorphism for any $H,K \leq G$. Thus
\[
\FFF: \Ho(\LSpG) \xrightarrow{- \wedge^{\mathbf{L}} X} \Ho(\C) \xrightarrow{\Psi^{-1}}(\LSpG)
\]
is an equivalence of categories, which proves our main rigidity theorem.

\section{Exotic models at odd primes}\label{sec:exotic}

Now let $p$ be an odd prime and $G$ a finite group such that $p$ does not divide the order of $G$. In this section we show that under these assumptions the triangulated category $\Ho(\LSpG)$
has an exotic model. This exotic model is an equivariant generalisation of Franke's model \cite{Fra96} and, like Franke's model, relies also on the ideas of Bousfield \cite{Bou85}.
Briefly, Franke's theorem shows the following: given a cohomological functor from a triangulated category (such as the homotopy category of a stable model category) into an abelian category, then under certain conditions this triangulated category can be realised as 
the derived category of twisted chain complexes. An important part of the assumptions is a relation between a ``split'' of the abelian category into smaller categories and the global injective dimension of this abelian category. 

\bigskip
We do not recall the conditions here and refer the reader to the original paper \cite{Fra96} for details. For more streamlined expositions see \cite{Roi08} or \cite{Pat16}. We also note that the proof in \cite{Fra96} contains a gap which is fixed for height one in \cite[Section 4]{Pat16}.

\subsection{The algebraic model} 
We will describe the abelian category that we are dealing with and study its global injective dimension.
The spectrum $E(1)$ defines a homology theory
$$E(1)_*(-): \Ho(L_1\Sp) \rightarrow E(1)_*E(1)\mbox{-}\Comod. $$
Here $E(1)_*E(1)\mbox{-}\Comod$ denotes the category of comodules over the flat Adams Hopf algebroid 
$$(E(1)_*,E(1)_*E(1)).$$ 
This category is an abelian category with enough injectives. Furthermore, it has global cohomological dimension equal to two \cite{Bou85}. 

In the equivariant context we need to consider \emph{Mackey functors} in comodules. Let $\mathcal{B}(G)$ denote the \emph{Burnside category} of $G$, see e.g. \cite[Section V.9]{LewMaySte86}. The category  $\mathcal{B}(G)$ is a (pre)additive category. The objects of $\mathcal{B}(G)$ are the cosets of the form $G/H$ for any subgroup $H \leq G$. Morphisms are additively generated by the restrictions, conjugations and transfers. 

A contravariant additive functor from $\mathcal{B}(G)$ to the category of abelian groups is called a Mackey functor. In this subsection we consider the category of contravariant additive functors from $\mathcal{B}(G)$ to the category of $E(1)_*E(1)$-comodules, i.e. Mackey functors in comodules. We will denote this category by
$$E(1)_*E(1)[\mathcal{M}(G)]\mbox{-}\Comod$$
and refer to its objects as \emph{comodule Mackey functors}.
Since equivariant stable homotopy groups carry the structure of a Mackey functor \cite[Section V.9]{LewMaySte86}, it follows that we have a homology theory
$$E(1)_*(-): \Ho(\LSpG) \rightarrow E(1)_*E(1)[\mathcal{M}(G)]\mbox{-}\Comod$$
which sends an object $X$ from $\Ho(\LSpG)$ to the comodule Mackey functor 
$$E(1)_*(X^{(-)}) : G/H \mapsto E(1)_*(X^H).$$
We now point out certain algebraic properties of the category of comodule Mackey functors. 

Let $M$ be a $p$-local $G$-Mackey functor (recall we assume that $p$ does not divide the order of $G$). For any $H \leq G$, let $\tau_H M$ denote the quotient of $M(H)$ by the subgroup generated by the images of transfers from proper subgroups of $H$. Then the action of the Weyl group $W_G(H)$ on $M(H)$ descends to the action of $W_G(H)$ on $\tau_H M$. Altogether we obtain a functor
\[\tau: \Fun(\mathcal{B}(G)^{op}, \mathbb{Z}_{(p)}\mbox{-}\Mod) \to  \prod_{(H) \leq G} \mathbb{Z}_{(p)}[W_G(H)]\mbox{-}\Mod,\]
sending a $p$-local Mackey functor $M$ to the collection $(\tau_H M)_{(H) \leq G}$. Here the notation $\Fun(\mathcal{B}(G)^{op}, \mathbb{Z}_{(p)}\mbox{-}\Mod)$ stands for the category of $p$-local Mackey functors.

The following fact is well-known and a proof can be found in \cite[Proposition III.4.20 and Theorem III.4.24]{Sch15} (see also \cite[Corollary 4.4]{Th88} and \cite[Appendix A]{GreMay95}).

\begin{proposition} \label{Mackeysplitting} Let $G$ be a finite group and $p$ a prime which does not divide the order of $G$. Then the functor
\[\tau: \Fun(\mathcal{B}(G)^{op}, \mathbb{Z}_{(p)}\mbox{-}\Mod) \to  \prod_{(H) \leq G} \mathbb{Z}_{(p)}[W_G(H)]\mbox{-}\Mod\]
is an equivalence of categories. \qed
\end{proposition} 

We get immediate corollaries:
\begin{corollary}
If $p$ does not divide the order of $G$, then 
$$E(1)_*E(1)[\mathcal{M}(G)]\mbox{-}\Comod \sim \prod_{(H) \leq G} E(1)_*E(1)[W_G(H)]\mbox{-}\Comod.$$
Here $E(1)_*E(1)[W_G(H)]\mbox{-}\Comod$ denotes the category of $E(1)_*E(1)$-comodules with $W_G(H)$-action.  \qed
\end{corollary} 

\begin{corollary} \label{dimcoro} Let $p$ be an odd prime and $G$ a finite group such that $p$ does not divide the order of $G$. Then the global cohomological dimension of the abelian category $E(1)_*E(1)[\mathcal{M}(G)]\mbox{-}\Comod$ is equal to two

\end{corollary}

\begin{proof} By the previous corollary it suffices to show that the cohomological dimension of $E(1)_*E(1)[W_G(H)]\mbox{-}\Comod$ is equal to two. Bousfield's result \cite{Bou85} tells us that the cohomological dimension of $E(1)_*E(1)\mbox{-}\Comod$ is two. Since the order of the group $W_G(H)$ is invertible in $E(1)_*E(1)\mbox{-}\Comod$, a proposition of Mitchell \cite[Proposition 3.4]{Mit68} completes the proof. \end{proof}

\subsection{Exotic equivalence}

In this subsection we apply Franke's theorem to construct an algebraic model for $\LSpG$ for $p$ odd and finite group $G$ such that $p$ does not divide the order of $G$. We will check that the conditions of Franke's theorem \cite{Fra96} are satisfied. Franke's theorem makes use of ``Eilenberg-MacLane objects'' for injective objects in order to set up an Adams spectral sequence, and we are going to construct those here.

\begin{lemma} \label{inj} Let $J$ be an injective $E(1)_*E(1)$-comodule. Then for any finite group $K$, the object $\mathbb{Z}[K] \otimes J$ is injective in $E(1)_*E(1)[K]\mbox{-}\Comod$.
\end{lemma}

\begin{proof} The functor 
$$\mathbb{Z}[K] \otimes - : E(1)_*E(1)\mbox{-}\Comod  \longrightarrow E(1)_*E(1)[K]\mbox{-}\Comod$$
is right adjoint (and also left adjoint) to the forgetful functor 
$$E(1)_*E(1)[K]\mbox{-}\Comod \longrightarrow E(1)_*E(1)\mbox{-}\Comod.$$
Additionally, the forgetful functor is exact. This implies the desired result. \end{proof}

Next, we show that certain injective objects in $E(1)_*E(1)[\mathcal{M}(G)]\mbox{-}\Comod$ are realisable.

\begin{proposition} \label{realizationinj} Let $p$ be an odd prime, $G$ a finite group and suppose that $p$ does not divide the order of $G$. Let $I$ be an injective object in $E(1)_*E(1)[\mathcal{M}(G)]\mbox{-}\Comod$ which under the equivalence of the previous subsection corresponds to 
$$(\mathbb{Z}[W_G(H)] \otimes J)_{(H) \leq G} \in  \prod_{(H) \leq G} E(1)_*E(1)[W_G(H)]\mbox{-}\Comod,$$
where $J$ is an injective $E(1)_*E(1)$-comodule. Then there exists a $v_1$-local $G$-spectrum $\mathcal{I} \in \Ho(\LSpG)$, such that the comodule Mackey functor $E(1)_*(\mathcal{I}^{(-)})$ is isomorphic to $I$.

\end{proposition}

\begin{proof} It follows from \cite[Proposition 8.2]{Bou85} that there is a $v_1$-local spectrum $\mathcal{J}$ such that $E(1)_*(\mathcal{J}) \cong J$. Recall $\mathcal{P}(H | G)$ denotes the family of subgroups in $G$ which are proper subconjugates of $H$. Further, recall that
$$\varepsilon^* : \Ho(L_1 \Sp) \rightarrow \Ho(\LSpG) $$
denotes the inflation functor associated to the projection $\varepsilon : G \rightarrow 1$. Define
$$\mathcal{I} = \bigvee_{(H) \leq G} G/H_{+} \wedge \widetilde{E}\mathcal{P}(H | G) \wedge \varepsilon^*(\mathcal{J}).$$
By the properties of geometric fixed points (see Section \ref{sec:localcat}), for any subgroup $H$ of $G$, we have an underlying $W_G(H)$-equivariant equivalence of $W_G(H)$-spectra
$$\Phi^H(\mathcal{I}) \simeq  W_G(H)_+ \wedge \mathcal{J}.$$
This implies that for any subgroup $H \leq G$,
$$E(1)_*(\Phi^H(\mathcal{I})) \cong  \mathbb{Z}[W_G(H)] \otimes J $$
inside the category $E(1)_*E(1)[W_G(H)]\mbox{-}\Comod$. Next, let $E(1)_*(\Phi)$ denote the functor sending $X$ to $$(E(1)_*(\Phi^H(X)))_{(H) \leq G} \in \prod_{(H) \leq G} E(1)_*E(1)[W_G(H)]\mbox{-}\Comod.$$ Then by \cite[Proposition III.4.28]{Sch15}, the diagram of categories
$$\xymatrix{\Ho(\LSpG)  \ar[r]^-{E(1)_*(-)} \ar[d]_-{E(1)_*(\Phi)} &   E(1)_*E(1)[\mathcal{M}(G)]\mbox{-}\Comod \ar[dl]^{\sim} \\ \prod_{(H) \leq G} E(1)_*E(1)[W_G(H)]\mbox{-}\Comod, & }$$
commutes up to a natural isomorphism. This commutative diagram together with the latter isomorphism and the definition of $I$ gives an isomorphism of comodule Mackey functors $E(1)_*(\mathcal{I}^{(-)}) \cong I$. \end{proof}

\bigskip
The following proposition implies that there is an Adams spectral sequence for calculating morphisms in $\Ho(\LSpG)$:

\begin{proposition} \label{SScondition} Let $p$ be an odd prime and $G$ a finite group such that $p$ does not divide the order of $G$. Then any comodule Mackey functor $$M \in E(1)_*E(1)[\mathcal{M}(G)]\mbox{-}\Comod$$ embeds into an injective comodule Mackey functor $I$ which satisfies the following properties:

{\rm (i)} There exists $\mathcal{I} \in \Ho(\LSpG)$ such that $E(1)_*(\mathcal{I}^{(-)})$ is isomorphic to $I$.

{\rm (ii)} The map
$$E(1)_*(-): [X, \mathcal{I}]^{G, v_1} \longrightarrow  \Hom_{E(1)_*E(1)[\mathcal{M}(G)]}(E(1)_*(X^{(-)}), E(1)_*(\mathcal{I}^{(-)})) $$
is an isomorphism for any $X$. Here $\Hom_{E(1)_*E(1)[\mathcal{M}(G)]}(-,-)$ denotes the Hom-group of comodule Mackey functors.

\end{proposition}

\begin{proof} Note that  by adjunction, every object of $E(1)_*E(1)[W_G(H)]\mbox{-}\Comod$ embeds into an injective object of the form $\mathbb{Z}[W_G(H)] \otimes J$, where $J$ is an injective $E(1)_*E(1)$-comodule. By choosing large enough $J$, we can embed  
the comodule Mackey functor $M$ into an injective $I$ which corresponds to $(\mathbb{Z}[W_G(H)] \otimes J)_{(H) \leq G}$ under the equivalence
$$E(1)_*E(1)[\mathcal{M}(G)]\mbox{-}\Comod \sim \prod_{(H) \leq G} E(1)_*E(1)[W_G(H)]\mbox{-}\Comod.$$
It follows from \cite[Proposition 8.2]{Bou85} that there exists a $v_1$-local spectrum $\mathcal{J}$, such that $E(1)_*(\mathcal{J})$ is isomorphic to $J$ and the map
$$E(1)_*(-): [Z, \mathcal{J}]_*^{L_1\Sp} \longrightarrow \Hom_{E(1)_*E(1)}(E(1)_*(Z), E(1)_*(\mathcal{J}))_* $$
is an isomorphism for any $v_1$-local spectrum $Z$ (here $\Hom(-,-)_*$ stands for graded homomorphisms). As in the previous proposition define
$$\mathcal{I} = \bigvee_{(H) \leq G} G/H_{+} \wedge \widetilde{E}\mathcal{P}(H | G) \wedge \varepsilon^*(\mathcal{J}).$$
The proof of the previous proposition implies Part (i). It remains to prove Part (ii).

We start by observing that by Proposition \ref{generators}, it suffices to show that the map
\[
E(1)_*(-):[L_{v_1}\SigmaGH, \mathcal{I}]^{G, v_1}_* \rightarrow \Hom_{E(1)_*E(1)[\mathcal{M}(G)]}(E(1)_*((\SigmaGH)^{(-)}), E(1)_*(\mathcal{I}^{(-)}))_*
\]
is an isomorphism for any subgroup $H \leq G$. Since $p$ does not divide the order of $G$, it follows that there is an isomorphism of comodule Mackey functors
$$E(1)_*((\SigmaGH)^{(-)}) \cong E(1)_*(S^0) \otimes [-, \SigmaGH]_0^G.$$
This isomorphism is induced by the $H$-equivariant derived unit map
$$S^0 \longrightarrow \Res^G_H(\SigmaGH)$$
and uses that $L_1B\Gamma_+$ is stably equivalent to $L_1S^0$ if $p$ does not divide the order of $\Gamma$. Now by adjunction, one sees that there is a commutative diagram
$$\xymatrix{[L_{v_1}\SigmaGH, \mathcal{I}]^{G, v_1}_* \ar[rr]^-{E(1)_*(-)}  \ar[d]^{\cong}  & &  \Hom_{E(1)_*E(1)[\mathcal{M}(G)]}(E(1)_*((\SigmaGH)^{(-)}), E(1)_*(\mathcal{I}^{(-)}))_* \ar[d]^{\cong}  \\ [L_1S^0, \mathcal{I}^H]^{L_1\Sp}_* \ar[rr]^-{E(1)_*(-)} & &    \Hom_{E(1)_*E(1)}(E(1)_*(S^0), E(1)_*({\mathcal{I}}^H))_*. }$$
Hence in order to complete the proof, it suffices to show that the lower horizontal map is an isomorphism in this diagram. Now properties of geometric fixed points, the double coset formula and the Wirthm\"uller isomorphism imply that the derived fixed points ${\mathcal{I}}^H$ are equivalent to 
$$\bigvee_{(K) \leq G} \\\\ \bigvee_{\substack{[g] \in H \setminus G /K,\\ H \cap {}^gK={}^gK }} \mathcal{J}$$
in $\Ho(L_1 \Sp)$. This completes the proof by definition of $\mathcal{J}$. \end{proof}

Now let $$\mathcal{C}^{([1],1)}(E(1)_*E(1)[\mathcal{M}(G)]\mbox{-}\Comod)$$ denote the model category of twisted $([1],1)$-chain complexes of comodule Mackey functors. We remind the reader that a twisted $([1],1)$-chain complex is a pair $(C, d)$, where $C$ is a comodule Mackey functor and $d : C \rightarrow C[1]$ is a morphism of comodule Mackey functors which is a differential, i.e. $d[1]\circ d=0$. Here $[1]$ denotes the shift of the grading. For the model structure on the category of twisted chain complexes for general abelian categories see \cite{Fra96, BarRoi11a}.

Let 
$$\mathcal{D}^{([1],1)}(E(1)_*E(1)[\mathcal{M}(G)]\mbox{-}\Comod)$$ 
denote the homotopy category of $\mathcal{C}^{([1],1)}(E(1)_*E(1)[\mathcal{M}(G)]\mbox{-}\Comod)$. A standard idempotent splitting argument together with Proposition \ref{SScondition} shows that the conditions of Franke's theorem are satisfied (see Theorem 4.2.4 and Theorem 4.2.5 of \cite{Pat16} for a complete proof). Hence we obtain:

\begin{theorem} \label{Frankeequiv} Let $p$ be an odd prime and $G$ a finite group. Suppose that $p$ does not divide the order of $G$. Then there is an equivalence of categories
$$\mathcal{D}^{([1],1)}(E(1)_*E(1)[\mathcal{M}(G)]\mbox{-}\Comod) \sim \Ho(\LSpG).$$
The equivalence is triangulated if $p \geq 5$.
\end{theorem}

This theorem provides an algebraic model for $\Ho(\LSpG)$ when $p$ is odd and does not divide the order of $G$. Moreover, this equivalence does not come from a zig-zag of Quillen equivalences since the homotopy types of mapping spaces in  $\LSpG$ are not products of Eilenberg-MacLane spaces in general. Hence the model category 
$$\mathcal{C}^{([1],1)}(E(1)_*E(1)[\mathcal{M}(G)]\mbox{-}\Comod)$$ 
provides an algebraic exotic model for $\LSpG$.

\subsection{A $G$-equivariant stable exotic model} The exotic model constructed in the previous subsection is algebraic but it is not $G$-equivariant in the sense of Section \ref{prelim}. In this subsection we explain why it is Quillen equivalent to a $G$-equivariant stable model category. This model will then provide a counterexample which shows that if $p$ is odd and does not divide the order of $G$, then the $p$-local version of Theorem \ref{maintheo} does not hold.

It follows from Morita theory \cite{SchShi03} and \cite[Proposition 6.3]{DugShi07} that there is an endomorphism differential graded algebra $B$ such that the model category $\mathcal{C}^{([1],1)}(E(1)_*E(1)\mbox{-}\Comod)$ of twisted $([1],1)$-chain complexes of $E(1)_*E(1)$-comodules is related by a zig-zag of Quillen equivalences to the model category $B$-Mod of differential graded $B$-modules. Then the model category $\mathcal{C}^{([1],1)}(E(1)_*E(1)[\mathcal{M}(G)]\mbox{-}\Comod)$ is Quillen equivalent to the model category of contravariant additive functors from $\mathcal{M}[G]$ to $B$-Mod equipped with the projective model structure. We denote this model category by $B[\mathcal{M}(G)]$-Mod. 

\bigskip
Let $HB$ be the associated (cofibrant) Eilenberg-MacLane ring spectrum of $B$ (see \cite{Shi07} and \cite[Theorem 7.1.2.13]{Lur16}). Now we will describe $B[\mathcal{M}(G)]$-Mod in terms of $HB$ and thus provide the desired exotic $G$-equivariant stable model.
By replacing $HB$ sufficiently cofibrantly and using the fact that the inflation functor
$$\varepsilon^*: \Sp \longrightarrow \SpG$$
is a symmetric monoidal left Quillen functor, we see that there is a cofibrant $G$-equivariant ring spectrum $\varepsilon^*HB$ such that its underlying orthogonal ring spectrum is $HB$. The exotic $G$-equivariant stable model we are looking for is the model category $\varepsilon^*HB$-Mod \cite[Section III.7]{ManMay}.

In what follows we will always assume that $p$ is an odd prime, $G$ a finite group and $p$ does not divide the order of $G$.

\begin{proposition} \label{inflationhomotopy} Let $X$ be a $p$-local, cofibrant orthogonal spectrum and $H$ a subgroup of $G$. Then the natural map induced by the inflation
$$A(H) \otimes \pi_*X \longrightarrow \pi_*^H(\varepsilon^*X) $$
is an isomorphism, where $A(H)$ is the Burnside ring of $H$. These isomorphisms are compatible with the Mackey structure.

\end{proposition}

\begin{proof} Both sides are homology theories in $X$. Thus we only have to check that the map is an isomorphism for the sphere spectrum $S^0_{(p)}$. This follows from the fact that $p$ does not divide the order of $G$. \end{proof}

\begin{corollary} \label{EMDGMorita} For any subgroups $K$ and $L$ of $G$, there are isomorphisms
\begin{align*}[G/K_+ \wedge \varepsilon^*HB, G/L_+ \wedge \varepsilon^*HB]_*^{\varepsilon^*HB} \cong  \bigoplus_{[g] \in K \setminus G /L}  A(K \cap {}^g L) \otimes H_*B \\ \cong [\SigmaGK, \SigmaGL]_0^G \otimes H_*B, \end{align*}
where $H_*B$ is the graded homology ring of $B$. \qed \end{corollary} 

\begin{proposition} \label{GTopmodel} The model categories $\mathcal{C}^{([1],1)}(E(1)_*E(1)[\mathcal{M}(G)]\mbox{-}\Comod)$ and $\varepsilon^*HB$-Mod are Quillen equivalent. 

\end{proposition} 

\begin{proof} Let $\mathcal{R}$ denote a cofibrant replacement of the derived endomorphism ring spectrum 
$$\mathbf{R}\Hom_{\varepsilon^*HB}(\vee_{H \leq G} (G/H_+ \wedge  \varepsilon^*HB), \vee_{H \leq G} (G/H_+ \wedge  \varepsilon^*HB)).$$ 
Then $\varepsilon^*HB$-Mod is Quillen equivalent to $\mathcal{R}$-Mod. Since $p$ does not divide the order of $G$, it follows from Corollary \ref{EMDGMorita} that there is a stable equivalence of cofbrant ring spectra
$$\xymatrix{ \mathbf{R}\Hom_G(\vee_{H \leq G} (\SigmaGH),  \vee_{H \leq G} (\SigmaGH)) \wedge^{\mathbf{L}} HB \ar[r]^-\sim & \mathcal{R}.}$$
Let $\mathcal{A}$ denote the ring spectrum $\mathbf{R}\Hom_G(\vee_{H \leq G} (\SigmaGH),  \vee_{H \leq G} (\SigmaGH))$ for short. There is canonical map of ring spectra
$$\mathcal{A} \wedge^{\mathbf{L}} HB \longrightarrow H\pi_0\mathcal{A} \wedge^{\mathbf{L}} HB \longrightarrow H\pi_0\mathcal{A} \wedge_{H\mathbb{Z}}^{\mathbf{L}} HB$$
which is again a weak equivalence because of Corollary \ref{EMDGMorita}. On the other hand it follows from \cite[Example 3.2.4.4. and Theorem 7.1.2.13]{Lur16} that the ring spectra 
\[
H\pi_0\mathcal{A} \wedge_{H\mathbb{Z}}^{\mathbf{L}} HB \,\,\,\,\,\mbox{and}\,\,\,\,\,H(\pi_0\mathcal{A} \otimes B)
\]
are weakly equivalent. The model category $H(\pi_0\mathcal{A} \otimes B)$-Mod is Quillen equivalent to $(\pi_0\mathcal{A} \otimes B)\mbox{-Mod}$ by \cite{Shi07} and \cite[Theorem 7.1.2.13]{Lur16}. Finally, we note that the category $(\pi_0\mathcal{A} \otimes B)\mbox{-Mod}$ is equivalent to $B[\mathcal{M}(G)]$-Mod, which is obtained by interpreting modules over a ring with several objects as modules over a usual ring in terms of idempotents. This completes the proof. \end{proof}

This proposition together with Theorem \ref{Frankeequiv} provides the desired counterexample of an exotic $G$-equivariant stable model. It is easy to trace through and see that orbits correspond to each other when passing to homotopy categories, and the Mackey structure is preserved. Thus, we arrived at a counterexample for $v_1$-local $G$-equivariant rigidity when $p \nmid |G|$ and $p \ge 5$.

\bibliographystyle{alpha}
\def\cprime{$'$}

\end{document}